\documentclass[12pt]{article}\usepackage{CJK}
\usepackage{amssymb}
\usepackage{amsthm}
\usepackage{amsmath}
\usepackage{mathrsfs}
\usepackage{graphicx}
\usepackage{epstopdf}
\usepackage{multirow}
\usepackage{longtable}
\usepackage{mathtools}
\usepackage{subfigure}
\usepackage{float}
\usepackage{geometry}
\usepackage{galois}
\usepackage{color}
\usepackage{diagbox}
\usepackage{empheq}
\usepackage{textcomp}
\usepackage[english]{babel}
\usepackage{framed}
\usepackage{bm}
\usepackage{appendix}
\usepackage[numbers,compress]{natbib}

\usepackage{fancyhdr}
\usepackage{tabularx}
\usepackage{booktabs}
\usepackage{dcolumn}
\usepackage{wrapfig}
\usepackage{multicol}
\usepackage{verbatim}

\makeatletter
\@addtoreset{equation}{section}
\makeatother
\definecolor{shadecolor}{gray}{0.85}
\newcommand{\arraystrech}[5]

\renewcommand{\arraystretch}{1.4}

\newcommand{\norm}[1]{\left\lVert#1\right\rVert}
\newcommand{\DF}[2]{ \displaystyle \frac{#1}{#2} }

\newtheorem{thm}{Theorem}
\newtheorem{lem}{Lemma}
\newtheorem{rem}{Remark}
\newtheorem{defn}{Definition}

\DeclareSymbolFont{largesymbol}{OMX}{yhex}{m}{n}
\DeclareMathAccent{\widehat}{\mathord}{largesymbol}{"62}

\begin{document}
\pagestyle{plain}

\title{A Fourier-Chebyshev Spectral Method for \\ Cavitation Computation
in Nonlinear Elasticity \thanks{The research was supported by the NSFC
projects 11171008 and 11571022.}}

\author{Liang Wei, \hspace{1mm} Zhiping Li\thanks{Corresponding author,
email: lizp@math.pku.edu.cn} \\ {\small LMAM \& School of Mathematical
Sciences, Peking University, Beijing 100871, China}}

\date{}

\maketitle

\begin{abstract}
A Fourier-Chebyshev spectral method is proposed in this paper for solving the
cavitation problem in nonlinear elasticity.
The interpolation error for the cavitation solution is analyzed, the elastic
energy error estimate for the discrete cavitation solution is obtained,
and the convergence of the method is proved.
An algorithm combined a gradient type method with a damped quasi-Newton method
is applied to solve the discretized nonlinear equilibrium equations.
Numerical experiments show that the Fourier-Chebyshev spectral method is
efficient and capable of producing accurate numerical cavitation solutions.
\end{abstract}


\section{Introduction}

In 1958, Gent and Lindley \cite{GentLindley1958} established the well known defective
model for the cavitation in nonlinear elasticity characterizing the phenomenon as
material instability associated to the dramatic growth of pre-existing micro voids
under large hydrostatic tensions, which very well matched the experimental observation
of sudden void formation in vulcanized rubber. Using the defective model, Gent et.al.
\cite{GentPark1984}, Lazzeri et.al. \cite{Bucknall1993,Bucknall1994}, and many other
researchers studied the cavitation phenomenon in elastomers containing rigid spherical
inclusions as well as in the standard model problems.
In 1982, Ball \cite{Ball1982} established the famous perfect model, in which cavitations
form in an originally intact body as an absolute energy minimizing bifurcation solution,
and produced the same cavitation criterion. The profound relationship of the two models
are studied by Sivaloganathan et.al. \cite{Sivaloganathan2002,Sivaloganathan2006} and
Henao \cite{Henao2009}.

Since the perfect model is known to be seriously challenged by the Lavrentiev phenomenon
\cite{Lavrentiv}, the defective model is chosen by most researchers in numerical studies
of the cavitation phenomenon, using mainly a variety of the finite element methods
(see Xu and Henao \cite{XuHenao}, Lian and Li \cite{LianLi2011dual,LianLi2011force},
Su and Li \cite{SuLi} among many others).
A spectral collocation method \cite{NegronMarrero}, which approximates the cavitation
solution with truncated Fourier series in the circumferential direction and finite differences
in the radial direction, is also found some success.

In a typical 2-dimensional defective model with a prescribed displacement boundary
condition, one considers to minimize the stored energy of the form
\begin{align}
\label{note:E}
E( \mathbf{u} ) = \int_{\Omega_\varepsilon}
	W( \nabla \mathbf{u} (\mathbf{x}) ) \mathrm{d} \mathbf{x} ,
\end{align}
in the set of admissible deformations
\begin{align}
\label{note:A_epsl}
\mathcal{A}_\varepsilon = \left\{
    \mathbf{u} \in W^{1,p}( \Omega_\varepsilon ) :
    \mathbf{u}~\text{is one-to-one}~ a.e.,~ \mathbf{u} \left|_{\partial \Omega}
    \right.=\mathbf{u}_0,~ \det \nabla \mathbf{u} > 0 ~a.e.
    \right\} ,
\end{align}
where $\Omega_\varepsilon = \Omega \setminus \overline{\mathbb{B}_\varepsilon
(\mathbf{x}_0 )} \subset \mathbb{R}^2$ is a domain occupied by the compressible
hyperelastic material in its reference configuration, with $\Omega$ being a
regular simply-connected domain and $\mathbb{B}_\varepsilon ( \mathbf{x}_0 )
=\{\mathbf{x} \in \mathbb{R}^2: |\mathbf{x}|< \varepsilon \}$ being a pre-existing
circular defect of radius $\varepsilon \ll 1$ centered at $\mathbf{x}_0$, and where
$W:\mathbb{M}_{+}^{2\times 2} \rightarrow \mathbb{R}^{+}$ is the stored energy
density function of the hyperelastic material, and $\mathbb{M}_{+}^{2\times 2}$
denotes the set of $2\times 2$ matrices with positive determinant.

The Euler-Lagrange equation of the above minimization problem is the following
displacement/traction boundary value problem:
\begin{equation}
\label{EL:original}
\begin{cases}
\mathrm{div} \displaystyle \frac{\partial W (\nabla \mathbf{u})}{\partial \nabla \mathbf{u}}
= \mathbf{0}, & \text{in}~ \Omega_\varepsilon ;  \vspace*{3mm} \\
\displaystyle \frac{\partial W (\nabla \mathbf{u})}{\partial \nabla \mathbf{u}} \cdot \mathbf{n}
= \mathbf{0}, &\text{on}~ \partial \mathbb{B}_\varepsilon( \mathbf{x}_0 );\vspace*{1mm} \\
\mathbf{u} (\mathbf{x}) = \mathbf{u}_0 (\mathbf{x}),
& \text{on}~\partial \Omega,
\end{cases}
\end{equation}
where $\mathbf{n}$ is the unit exterior normal with respect to $\Omega_\varepsilon$.

In the present paper, without loss of generality \cite{Ball1982,PietroStefania2013},
we consider the stored energy density function $W(\cdot)$ of the form
\begin{equation}
\label{note:W}
W( \nabla \mathbf{u} ) = \kappa | \nabla \mathbf{u} |^p
    + h( \det \nabla \mathbf{u} ), \qquad
    \nabla\mathbf{u}\in \mathbb{M}_{+}^{2\times 2},~ 1<p<2 ,
\end{equation}
where $\kappa$ is a positive material constant, $|\cdot |$ denotes the Frobenius norm
of a matrix and $h \in \mathcal{C}^3( (0,+\infty) )$ is a strictly convex function satisfying
\begin{equation}
\label{note:h}
h(t)\rightarrow +\infty ~\text{as}~ t \rightarrow 0^+, ~\text{and}~
\frac{h(t)}{t} \rightarrow +\infty ~\text{as}~ t \rightarrow +\infty .
\end{equation}
Since the cavitation solution is generally considered to have high regularity except
in a neighborhood of the defects, where the material experiences large expansion dominant
deformations, we restrict ourselves to a simplified reference configuration
$\Omega_{(\varepsilon,\gamma)} = \mathbb{B}_\gamma( \mathbf{0}) \setminus
\overline{\mathbb{B}_\varepsilon( \mathbf{0}) }$ $(0<\varepsilon \ll \gamma \leq 1)$, and denote
\begin{equation}\begin{split}
\label{note:A_epsl_gamma}
\mathcal{A}_{(\varepsilon,\gamma)}( \mathbf{u}_0 ) = \{
    \mathbf{u} \in W^{1,p}( \Omega_{(\varepsilon,\gamma)} ) :~
\mathbf{u}~\text{is one-to-one}~ a.e.,~&
    \mathbf{u} \left|_{\partial \mathbb{B}_\gamma( \mathbf{0})} \right.
        = \mathbf{u}_0, \\
& \det \nabla \mathbf{u} > 0 ~a.e. \} .
\end{split}\end{equation}

Taking the advantages of the smoothness of the cavitation solutions in the defective
model when $\mathbf{u}_0$ is sufficiently smooth and the high efficiency and accuracy
of spectral methods in approximating smooth solutions of partial differential equations
(see Li and Guo \cite{LiGuo}, Shen \cite{Shen,ShenTang} etc.),
we develop a Fourier-Chebyshev spectral method to solve the
Euler-Lagrange equation \eqref{EL:original}, which approximates the cavitation solution
with truncated Fourier series in the circumferential direction and truncated Chebyshev
series in the radial direction. The interpolation error for the cavitation solution is
analyzed, the elastic energy error estimate for the discrete cavitation solution is
derived, and the convergence of the method is proved.
An algorithm combined a gradient type method with a damped quasi-Newton method
is applied to solve the discretized nonlinear equilibrium equations.
Numerical experiments show that the Fourier-Chebyshev spectral method is efficient
and capable of producing highly accurate numerical cavitation solutions.
We would like to point out here, even though the reference domain is restricted to
a circular ring $\Omega_{(\varepsilon,\gamma)}$, to further exploring its highly
efficient feature in a neighborhood of a cavity surface, our method can be coupled
with a domain decomposition method, especially in combining with some finite element
methods to extend the application to more general situations
with multiple pre-existing tiny voids.

The structure of the rest of the paper is as follows.
In \S 2, we rewrite the Euler-Lagrange equation of the cavitation problem in
a proper computing coordinates. In \S 3, the Fourier-Chebyshev spectral method is
applied, the corresponding discrete equilibrium equation is derived, and an algorithm
to solve the nonlinear equation is presented. \S 4 is devoted to
the analysis of the interpolation error of the cavitation solution, the elastic energy
error bound and the convergence of the discrete cavitation solution.
In \S 5, numerical experiments and results are presented to show the efficiency
and accuracy of our method.

\section{The Euler-Lagrange Equation}

In the Cartesian coordinate system, an admissible deformation
$\mathbf{u} \in \mathcal{A}_{(\varepsilon,\gamma)}(\mathbf{u}_0)$ is written as $\mathbf{u}( \mathbf{x} )
= \left[ u_1(x_1,x_2), u_2(x_1,x_2) \right]^T $. Denote
\begin{align}
\label{note:DFg}
D(\mathbf{u}) := \det \nabla \mathbf{u}, \quad
F(\mathbf{u}) := \frac{1}{2} |\nabla \mathbf{u}|^2, \quad
g(t) := \kappa \left( \sqrt{2t} \right)^p,
\end{align}
and to futher simplify the notation, $D(\mathbf{u})$ and $F(\mathbf{u})$ will be denoted
below as $D$, $F$ wherever no ambiguity is caused.
For the elastic energy density function $W(\cdot)$ given by \eqref{note:W} and the elastic energy $E(\cdot)$ given by \eqref{note:E}, we have
\begin{align}
\label{note:E_u}
E( \mathbf{u} ) = \int_{\Omega_{(\varepsilon,\gamma)}}
    \left[ g\left( F(\mathbf{u}) \right) +
            h\left( D(\mathbf{u}) \right) \right] \mathrm{d} \mathbf{x} .
\end{align}

For the convenience of the implementation of the Fourier-Chebyshev spectral method,
we introduce a $(\rho,\phi)$-coordinate system defined on the computational domain
$\Omega' :=(-1,1) \times (0,2\pi)$, by coupling the Cartesian to polar coordinates transformation
\begin{equation}
\begin{cases}\label{coordinate:polar}
x_1 = r \cos \theta, \\
x_2 = r \sin \theta,
\end{cases} \text{and} \quad
\begin{cases}
u_1 = R(r,\theta) \cos \Theta(r,\theta), \\
u_2 = R(r,\theta) \sin \Theta(r,\theta),
\end{cases}
\end{equation}
defined on the domain $(\varepsilon,\gamma) \times (0,2\pi)$, with a
transformation defined by
\begin{equation}
\begin{cases}\label{coordinate:rho_phi}
r = \frac{\gamma+\varepsilon}{2}+\frac{\gamma-\varepsilon}{2} \rho,\\
\theta = \phi,
\end{cases} \text{and} \quad
\begin{cases}
R(r,\theta) = P(\rho,\phi), \\
\Theta(r,\theta) = Q(\rho,\phi)+\phi,
\end{cases}
\end{equation}
defined on the computational domain
$\Omega' =(-1,1) \times (0,2\pi)$.

In $(\rho,\phi)$-coordinates, $D(\mathbf{u})= \det \nabla \mathbf{u}$,
$F(\mathbf{u})= |\nabla \mathbf{u}|^2 /2$ defined in \eqref{note:DFg} can be rewritten as functions of $P(\rho,\phi)$, $Q(\rho,\phi)$:
\begin{subequations}
\label{note:DF}
\begin{align}
\label{note:D}
D(P,Q) &= \frac{\rho_r}{r} P \left[ P_\rho (Q_\phi+1) - P_\phi Q_\rho \right] ,\\
F(P,Q) &= \frac{\rho_r^2}{2} ( P_\rho^2 + P^2 Q_\rho^2 )
    + \frac{1}{2 r^2} \left[ P_\phi^2 + P^2 (Q_\phi+1)^2 \right] ,
\end{align}
\end{subequations}
where $\rho_r = 2/(\gamma-\varepsilon)$;
the elastic energy $E(\mathbf{u})$ in \eqref{note:E_u} can be expressed as
\begin{align}
\label{energy}
E( P,Q ) = \int_{\Omega'}
	\left[ g\left( F(P,Q) \right) + h\left( D(P,Q) \right) \right]
    \frac{r}{\rho_r} \mathrm{d} \rho \mathrm{d} \phi,
\end{align}
and the set of admissible deformation $\mathcal{A}_{(\varepsilon,\gamma)}( \mathbf{u}_0 )$
(see \eqref{note:A_epsl_gamma}) is reformulated as
\begin{equation}\begin{split}
\mathcal{A}_{\Omega'}( \mathbf{u}_0 ) = \{ (P,Q):~
\exists~ \mathbf{u} \in \mathcal{A}_{(\varepsilon,\gamma)}( \mathbf{u}_0 ), \,
&   \text{{\it s.t.} }
    \mathbf{u}~\text{is mapped to}~(P,Q) \\
& \text{by the transformations \eqref{coordinate:polar} and
    \eqref{coordinate:rho_phi}} \}.
\end{split}\end{equation}
Thus, in $(\rho,\phi)$-coordinates, the cavitation solution
$(P,Q)\in \mathcal{A}_{\Omega'}( \mathbf{u}_0 )$ is
characterized as the minimizer
of $E( P,Q )$ in $\mathcal{A}_{\Omega'}( \mathbf{u}_0 )$, {\it i.e.}
\begin{align}
\label{problem}
(P,Q)=\arg\min_{ (P,Q)\in \mathcal{A}_{\Omega'}( \mathbf{u}_0 ) } E(P,Q),
\end{align}
or alternatively, as the solution to the Euler-Lagrange equation of \eqref{problem}:
\begin{align}
\label{EL:body}
\left\{\begin{aligned}
    \int_{\Omega'} f_1(P,Q;\bar{P}) \mathrm{d} \rho \mathrm{d} \phi=0,\\
    \int_{\Omega'} f_2(P,Q;\bar{Q}) \mathrm{d} \rho \mathrm{d} \phi=0,
\end{aligned}\right. \quad
\forall~(\bar{P},\bar{Q}) \in \mathcal{A}_{\Omega'}( \mathbf{0} ),
\end{align}
where, by the definition and direct calculations, we have
\begin{align*}
f_1 &:= \bar{P} \left[
    g'(F)\left( r\rho_r P Q_\rho^2 + \frac{1}{r\rho_r} P (Q_\phi+1)^2 \right)
    + h'(D) \left( P_\rho (Q_\phi+1) - P_\phi Q_\rho \right) \right] \\
&\qquad + \bar{P}_\rho \left[ g'(F) r\rho_r P_\rho
        + h'(D) P (Q_\phi+1) \right]
    + \bar{P}_\phi \left[ g'(F) \frac{1}{r\rho_r}  P_\phi
        - h'(D) P Q_\rho \right] ,\\
f_2 &:= \bar{Q}_\rho P \left[ g'(F) r\rho_r P Q_\rho
        - h'(D) P_\phi \right]
    + \bar{Q}_\phi P \left[ g'(F) \frac{1}{r\rho_r} P (Q_\phi+1)
        + h'(D) P_\rho \right].
\end{align*}

\section{The Fourier-Chebyshev Spectral Method}

To discretize the Euler-Lagrange equation \eqref{EL:body} defined on
$\Omega' =(-1,1) \times (0,2\pi)$ in $(\rho,\phi)$-coordinates, we first approximate
the unknowns $(P(\rho,\phi), Q(\rho,\phi))$ by the finite Fourier-Chebyshev polynomials:
\begin{subequations}
\label{note:PQ_NM}
\begin{align}
P^{NM}(\rho,\phi) &:= \sum_{j=0}^M \left(
    \sum_{k=0}^{N/2} \alpha_{k,j} \cos k\phi +
    \sum_{k=1}^{N/2-1} \beta_{k,j} \sin k\phi \right) T_j(\rho), \vspace*{2mm} \\
\label{note:Q_NM}
Q^{NM}(\rho,\phi) &:= \sum_{j=0}^M \left(
    \sum_{k=0}^{N/2}   \xi_{k,j}  \cos k\phi +
    \sum_{k=1}^{N/2-1} \eta_{k,j} \sin k\phi \right) T_j(\rho),
\end{align}
\end{subequations}
where $T_j$ is the Chebyshev polynomial of the first kind of degree $j$, defined as
\[ T_j(x) = \cos(j \arccos x), \]
with $T_0(x) = 1, T_1(x) = x$ and satisfying the recurrence relation \cite{ShenTang}
\[ T_{j+1}(x) = 2x T_j(x) - T_{j-1}(x), \quad j\geq 1. \]

\begin{rem}
We use the trigonometric polynomials to approximate $Q=\Theta-\theta$ instead of
$\Theta$ (see \eqref{coordinate:rho_phi} and \eqref{note:Q_NM}) so that the
Gibbs phenomenon can be avoided (see e.g. \cite{Gibbs}), since the
periodic extension of $Q=\Theta-\theta$ from $[0, 2\pi)$ to $\mathbb{R}^1$ is smooth,
while that of $\Theta$ is a sawtooth function with jump discontinuities at $2k\pi$,
$k=0,1,\dotsc$.
\end{rem}

The discretized problem of solving the Euler-Lagrange equation \eqref{EL:body} is then read as:
find $(P^{NM},Q^{NM}) \in \mathcal{B}^{NM}$ such that
\begin{align}
\label{EL:final}
\left\{\begin{aligned}
   \int_{\Omega'} f_1(P^{NM},Q^{NM};\bar{P})
        \mathrm{d} \rho \mathrm{d} \phi=0,\\
   \int_{\Omega'} f_2(P^{NM},Q^{NM};\bar{Q})
        \mathrm{d} \rho \mathrm{d} \phi=0,
\end{aligned}\right.
\quad \forall~(\bar{P},\bar{Q}) \in \mathcal{B}^{NM}_0,
\end{align}
where $\mathcal{B}^{NM}$ and $\mathcal{B}^{NM}_0$ are the discrete trial and
test function spaces defined as
\begin{subequations}
\label{note:space}
\begin{align}
\mathcal{B}^{NM} &:= \left\{ (P^{NM},Q^{NM}):\text{the Fourier-Chebyshev polynomials
\eqref{note:PQ_NM} satisfying} \right. \nonumber\\
&\qquad \left. P^{NM}(1,\phi_n) = P_0(1,\phi_n),~
    Q^{NM}(1,\phi_n) = Q_0(1,\phi_n),~ 0\leq n\leq N-1 \right\}, \label{note:B_NM} \\
\mathcal{B}^{NM}_0 &:= \left\{ (P^{NM},Q^{NM}):\text{the Fourier-Chebyshev polynomials
\eqref{note:PQ_NM} satisfying} \right. \nonumber\\
&\qquad\qquad\qquad\qquad\;\; \left. P^{NM}(1,\phi_n) = 0,~
    Q^{NM}(1,\phi_n) = 0,~ 0\leq n\leq N-1 \right\},
\end{align}
\end{subequations}
where, in \eqref{note:space}, $\phi_n = 2\pi n/N,~ 0\leq n\leq N-1$,
and the Dirichlet boundary condition $(P_0(1,\phi_n), Q_0(1,\phi_n))$ is defined by
$\mathbf{u}_0$ via the coordinates transformations \eqref{coordinate:polar} and
\eqref{coordinate:rho_phi}.
To solve the equation \eqref{EL:final} numerically, we need to replace the integrals
in \eqref{EL:final} by proper numerical quadratures.
Let $\left\{ \rho_{m'}, \omega^C_{m'} \right\}_{m'=0}^{M'}$ and $\left\{ \phi_{n'},
\omega^F_{n'} \right\}_{n'=0}^{N'-1}$ be the sets of Gauss-Chebyshev and Fourier
quadrature nodes and weights respectively, {\it i.e.} \cite{ShenTang}
\begin{align}
\begin{array}{lll}
\rho_{m'} = \cos \DF{(2m'+1) \pi}{2M'+2}, & \omega^C_{m'} = \DF{\pi}{M'+1},
    & 0\leq m'\leq M', \\
\phi_{n'} = \DF{2\pi n'}{N'}, & \omega^F_{n'} = \DF{2\pi}{N'},
    & 0\leq n'\leq N'-1,
\end{array}
\end{align}
then we are led to the following discretized Euler-Lagrange equation:
find $(P^{NM},Q^{NM}) \in \mathcal{B}^{NM}$
such that for all $(\bar{P},\bar{Q}) \in \mathcal{B}^{NM}_0$
\begin{align}
\label{EL:final_D}
\left\{\begin{aligned}
   \sum_{n'=0}^{N'-1} \sum_{m'=0}^{M'} f_1(P^{NM}(\rho_{m'}, \phi_{n'}),Q^{NM}(\rho_{m'},
   \phi_{n'});\bar{P}(\rho_{m'}, \phi_{n'})) \sqrt{1-\rho^2_{m'}}
        \omega^C_{m'} \omega^F_{n'}=0,\\
   \sum_{n'=0}^{N'-1} \sum_{m'=0}^{M'} f_2(P^{NM}(\rho_{m'}, \phi_{n'}),Q^{NM}(\rho_{m'},
   \phi_{n'});\bar{Q}(\rho_{m'}, \phi_{n'})) \sqrt{1-\rho^2_{m'}}
        \omega^C_{m'} \omega^F_{n'}=0.
\end{aligned}\right.
\end{align}

Let $\{a_k,b_k\}$ and $\{c_k,d_k\}$ be the discrete Fourier coefficients
of $P_0(1,\phi)$ and $Q_0(1,\phi)$ respectively, then
the boundary condition in \eqref{note:B_NM} can be expressed as
\begin{align}
\label{note:alpha_k0}
\begin{array}{cc}
    \begin{aligned}
    \alpha_{k,0} &= -\sum_{j=1}^M \alpha_{k,j} + a_k,\\
    \beta_{k,0}  &=  -\sum_{j=1}^M \beta_{k,j} + b_k,
    \end{aligned}
    &
    \begin{aligned}
    \xi_{k,0} &= -\sum_{j=1}^M \xi_{k,j} + c_k,   \quad 0 \leq k \leq N/2 ,\\
    \eta_{k,0} &= -\sum_{j=1}^M \eta_{k,j} + d_k, \quad 1 \leq k \leq N/2-1.
    \end{aligned}
\end{array}
\end{align}
Noticing also that the following $N \times M$ functions
\begin{equation}
\label{note:basis}
\left\{ \cos k\phi \cdot (T_{j}(\rho)-1)
    \right\}_{1\leq j\leq M}^{0\leq k\leq N/2} ,\quad
\left\{ \sin k\phi \cdot (T_{j}(\rho)-1)
    \right\}_{1\leq j\leq M}^{1\leq k\leq N/2-1},
\end{equation}
form a set of
bases for $\mathcal{B}^{NM}_0$, we conclude that the discrete Euler-Lagrange
equation \eqref{EL:final_D} consists of $2NM$ nonlinear algebraic equations,
which, for the simplicity of the notations, will be denoted as
$\mathbf{f}( \mathbf{y} ) = \mathbf{0}$, with $2NM$ unknowns $\mathbf{y}=$
$\left\{ \alpha_{k,j}, \xi_{k,j} \right\}_{1 \leq j \leq M}^{0 \leq k \leq N/2} \cup$
$\left\{ \beta_{k,j}, \eta_{k,j} \right\}_{1 \leq j \leq M}^{1 \leq k \leq N/2-1}$.
Denote $E(\mathbf{y})$ as the discrete elastic energy defined by replacing the
integral in $E(P^{NM},Q^{NM})$ (see \eqref{energy}) with the numerical quadrature,
then $\mathbf{f}( \mathbf{y} )$ may be viewed as the gradient of the discrete
elastic energy $E(\mathbf{y})$.

In our numerical experiments, the discrete equilibrium equations
$\mathbf{f}( \mathbf{y} ) = \mathbf{0}$, {\it i.e.} \eqref{EL:final_D},
are solved by an algorithm combined a gradient type method with a damped quasi-Newton
method \cite{quasiNewton}. More specifically, we use a gradient type method,
which calculates a descent direction of the energy and conducts a incomplete
line search in each iteration, to provide an appropriate initial cavity deformation for
a damped quasi-Newton method with Broyden's correction, which will then
produce a reasonably accurate numerical cavity solution.
The algorithm is summarized as follows, where the determinant of
the deformation $(P^{NM},Q^{NM})$ corresponding to $\mathbf{y}$ is denoted as
$D(\mathbf{y}) := D(P^{NM},Q^{NM})$ (see \eqref{note:D}).

\vspace*{2mm}
{\bf Algorithm:}
\begin{description}
  \item[Step 1] Given $\mathbf{y}^G_0$, set $TOL = 10^{-1}$, compute
  $\mathbf{f}(\mathbf{y}^G_0)$ and $E(\mathbf{y}^G_0)$.
  \item[Step 2] If $TOL < 10^{-10}$, then output $\mathbf{y}^G_0$ and stop;
  else, set $t^G_{-1} = 1$ and $j:=0$.
  \item[Step 3] For $j\geq 0$, if $|\mathbf{f}(\mathbf{y}^G_j)| < TOL$,
  then go to Step 6; else, set $t^G_j = 4\cdot t^G_{j-1}$.
  \item[Step 4] Set $\mathbf{y}^G_{j+1} = \mathbf{y}^G_{j} - t^G_j \cdot
  \mathbf{f}(\mathbf{y}^G_j)$, compute $\mathbf{f}(\mathbf{y}^G_{j+1})$,
  $E(\mathbf{y}^G_{j+1})$ and $D(\mathbf{y}^G_{j+1})$.
  \item[Step 5] If $t^G_j < 10^{-16}$, then output $\mathbf{y}^G_j$ and stop;
      else if $E(\mathbf{y}^G_{j+1}) < E(\mathbf{y}^G_{j})$ and
      $D( \mathbf{y}^G_{j+1} ) >0$, then set $j:=j+1$ and go to Step 3;
      else, set $t^G_j := t^G_j /2$ and go to Step 4.
  \item[Step 6] Set $\mathbf{y}^N_0 := \mathbf{y}^G_j$, compute
  $\mathbf{f}(\mathbf{y}^N_0)$ and $\mathbf{B}_0 = [
  \nabla \mathbf{f}( \mathbf{y}^N_0 ) ]^{-1}$, set $t^N_{-1} = 1$ and $k:=0$.
  \item[Step 7] For $k\geq 0$, if $|\mathbf{f}(\mathbf{y}^N_k)| < 10^{-10}$,
  then output $\mathbf{y}^N_k$ as the solution and stop;
      else, set $t^N_k = 4\cdot t^N_{k-1}$.
  \item[Step 8] Set $\mathbf{y}^N_{k+1} = \mathbf{y}^N_{k} - t^N_k \
  cdot \mathbf{B}_k \cdot \mathbf{f}(\mathbf{y}^N_k)$, compute
  $\mathbf{f}(\mathbf{y}^N_{k+1})$ and $D(\mathbf{y}^N_{k+1})$.
  \item[Step 9] If $t^N_k < 10^{-16}$, then go to Step 2 with
  $\mathbf{y}^G_0 := \mathbf{y}^N_k$ and $TOL := TOL/10$;
      else if $|\mathbf{f}(\mathbf{y}^N_{k+1})| <
      |\mathbf{f}(\mathbf{y}^N_{k})|$ and $D(\mathbf{y}^N_{k+1}) >0$, then go to Step 10;
      else, set $t^N_k := t^N_k /2$ and go to Step 8.
  \item[Step 10] Compute $\mathbf{s}_k = \mathbf{y}^N_{k+1} -
  \mathbf{y}^N_k$, $\mathbf{z}_k = \mathbf{f}(\mathbf{y}^N_{k+1}) -
  \mathbf{f}(\mathbf{y}^N_k)$, and
      \begin{align}
        \mathbf{B}_{k+1} = \mathbf{B}_k + \frac{
        \left( \mathbf{s}_k - \mathbf{B}_k \mathbf{z}_k \right)
        \mathbf{s}_k^T \mathbf{B}_k }{ \mathbf{s}_k^T \mathbf{B}_k \mathbf{z}_k}.
      \end{align}
      Set $k:=k+1$ and go to Step 7.
\end{description}

\section{Error Analysis and the Convergence Theorem}

In this section, we analyze the interpolation error of the discrete Fourier-Chebyshev spectral method for the cavitation solutions, which will enable us to
derive the elastic energy error estimate for the discrete cavitation solution, and
prove the convergence of the method.

Before analyzing the interpolation error of a cavitation solution,
we first introduce some notations.
Let $\mathcal{B} := \{ (P,Q) \in C^1(\overline{\Omega'}): P(1,\phi)=P_0(1,\phi),~
Q(1,\phi)=Q_0(1,\phi)\}$, let $\mathcal{B}_{+} := \{ (P,Q) \in \mathcal{B}:
D(P,Q)>0\}$ (see \eqref{note:D}), and denote $\mathcal{B}^{NM}_+=\mathcal{B}^{NM}\cap \mathcal{B}_+$ (see \eqref{note:B_NM}).
Let $\omega(\rho) := (1-\rho^2)^{-1/2}$, $\Lambda := (0,2\pi)$, $I := (-1,1)$, and recall that $(\rho,\phi) \in \Omega' = I \times \Lambda$. For given integers $\sigma \geq 0$ and
$\mu \geq 0$, denote $H^\sigma_\omega(I) = \{ \psi: \norm{\psi}_{ H^\sigma_\omega(I) }
< \infty \}$ the weighted Hilbert space with the norm defined as
\begin{align*}
\norm{\psi}_{H^\sigma_\omega(I)} = \left( \sum^\sigma_{j=0} \int_I
        \left| \frac{d^j \psi}{d \rho^j} \right|^2 \omega
        \mathrm{d} \rho \right)^{1/2},
\end{align*}
and denote $H^\mu( \Lambda; H^\sigma_\omega(I) )= \{ v: \norm{v}_{ H^\mu( \Lambda;
H^\sigma_\omega(I) ) } < \infty \}$ the Hilbert space equipped with the norm defined as
\begin{align*}
\norm{v}_{ H^\mu( \Lambda; H^\sigma_\omega(I) ) } = \left( \sum^\mu_{k=0} \int_\Lambda
        \norm{ \frac{\partial^k v}{\partial \phi^k} }_{H^\sigma_\omega(I)}^2
        \mathrm{d} \phi \right)^{1/2}.
\end{align*}

\begin{defn}
\label{defn:INM}
Define the interpolation operator $I^{NM}: \mathcal{B} \rightarrow \mathcal{B}^{NM}$ as
$$
[I^{NM} (P, Q)](\rho_m,\phi_n)=(P,Q)(\rho_m,\phi_n), \;\;
\forall\; 0\leq n \leq N-1,~0\leq m\leq M,
$$
where $\rho_m = \cos( m\pi / M )$, $\phi_n=2n \pi /N$.
\end{defn}

The interpolation operator $I^{NM}$ is shown to have the following error
estimates (see Lemma~5 in \cite{LiGuo}).
\begin{lem} \cite{LiGuo}
\label{lem:Liguo}
If $v \in H^{\beta}(\Lambda;H^\sigma_\omega(I)) \cap H^{\mu}(\Lambda;H^\alpha_\omega(I))
\cap H^{\mu'}(\Lambda;H^{\sigma'}_\omega(I)),~ 0\leq \alpha \leq \sigma,\sigma',~
0\leq \beta \leq \mu,\mu',~ \sigma,\sigma' >\frac{1}{2} ~\text{and}~ \mu,\mu' >1$,
then, there exists a constant $c>0$
independent of $P$, $Q$, $M$ and $N$, such that
\begin{equation*}\begin{split}
\norm{I^{NM}v-v}_{H^{\beta}(\Lambda;H^\alpha_\omega(I))}
&\leq cM^{2\alpha-\sigma} \norm{v}_{H^{\beta}(\Lambda;H^\sigma_\omega(I))}
    + cN^{\beta-\mu} \norm{v}_{H^{\mu}(\Lambda;H^\alpha_\omega(I))} \\
&\qquad \qquad +cq(\beta) M^{2\alpha-\sigma'} N^{\beta -\mu'}
    \norm{v}_{H^{\mu'}(\Lambda;H^{\sigma'}_\omega(I))},
\end{split}\end{equation*}
where $q(\beta)=0$ for $\beta>1$ and $q(\beta)=1$ for $\beta \leq 1$.
\end{lem}

\begin{thm}
\label{thm:PQ}
Let $(P,Q)\in \mathcal{B} \cap H^s( \Lambda; H^l_\omega(I) )$, $l>2$, $s>1$.
Then there exists a constant $c>0$ independent of $P$, $Q$, $M$ and $N$, such that
\begin{align*}
\norm{I^{NM}P - P}_{H^{\beta}(\Lambda;H^\alpha_\omega(I))}
&\leq c \|P\|_* \left( M^{2\alpha-l} + N^{\beta-s} \right),\\
\norm{I^{NM}Q - Q}_{H^{\beta}(\Lambda;H^\alpha_\omega(I))}
&\leq c \|Q\|_* \left( M^{2\alpha-l} + N^{\beta-s} \right),
\end{align*}
where $\alpha,\beta=0,1$ and $\|\cdot\|_*$ is a norm defined by:
\begin{align*}
\|v\|_* &:= \max\left\{
    \norm{v}_{H^\beta( \Lambda,H^l_\omega(I) ) },
    \norm{v}_{H^s    ( \Lambda,H^\alpha_\omega(I) ) },
    \norm{v}_{H^s    ( \Lambda,H^l_\omega(I) ) }
    \right\}.
\end{align*}
Furthermore, if $(P,Q)\in \mathcal{B} \cap
H^{l}(\Lambda;H^{l}_\omega(I))$ with $l> 6$, then
\begin{equation}\begin{split}\label{det_error}
&\| D(I^{NM}P,I^{NM}Q) -D(P,Q) \|_{C(\overline{\Omega'})} \\
&\qquad \qquad \le c (\|P\|_{H^{l}(\Lambda;H^{l}_\omega(I))}
    +\|Q\|_{H^{l}(\Lambda;H^{l}_\omega(I))}) (M^{6-l} + N^{3-l}).
\end{split}\end{equation}
\end{thm}

\begin{proof}
The first half of the theorem is a direct consequence of Lemma \ref{lem:Liguo} by
setting $\sigma=\sigma'=l$, $\mu=\mu'=s$, and taking $\alpha=0$ or $1$ and $\beta=0$
or $1$ respectively.

By \eqref{note:D}, the error estimate \eqref{det_error} follows from Lemma~\ref{lem:Liguo}
with $\sigma=\sigma'=\mu=\mu'=l$, $\alpha=\beta=3$, and the fact that
$H^3_\omega(\Omega') \hookrightarrow H^3(\Omega') \hookrightarrow C^1(\overline{\Omega'})$.
\end{proof}

In what follows below, we always assume that, for a cavitation solution $(P,Q)$, the following
hypotheses hold:
\begin{description}
  \item[(H1)] $(P,Q)\in \mathcal{B}_+ \cap H^{l}(\Lambda;H^{l}_\omega(I))$ with $l> 6$ and
             $(P,Q)$ is the energy minimizer in $\mathcal{B}_+$.
  \item[(H2)] there exists constants $c_F>1$ and $c_D>1$ such that $F$ and $D$
  (see \eqref{note:DF}) satisfies
\begin{align*}
c_F^{-1} \leq 2r^2 F \leq c_F, \quad
c_D^{-1} \leq D \leq c_D, \quad
\text{on}~ \overline{\Omega'}.
\end{align*}
\end{description}
\begin{rem}
Notice that, by \eqref{note:DF}
\begin{subequations}
\label{note:rDrF}
\begin{align}
rD    & = \rho_r P \left[ P_\rho (Q_\phi+1) - P_\phi Q_\rho \right] ,\\
2r^2F & = r^2 \rho_r^2 ( P_\rho^2 + P^2 Q_\rho^2 ) + P_\phi^2 + P^2 (Q_\phi+1)^2,
\end{align}
\end{subequations}
and for a cavitation solution $P\ge P_0>0$, and in the radially symmetric case
$Q_{\phi}=0$. the hypothesis $c_F^{-1} \leq 2r^2 F$ is not too harsh a requirement
on a general solution. While the other bounds are the direct consequences of
$(P,Q) \in C^1(\overline{\Omega'})$.
\end{rem}

To estimate the error on the elastic energy of the interpolation function of
a cavitation solution, we will making use of an auxiliary grid in radial direction
on which the elastic energy of the cavitation solution is radially
quasi-equi-distributed in the sense given in Lemma~\ref{lem:grid}. The properties
of such grids are given by the following two lemmas.

\begin{lem} \label{lem:energy}
Let $(P,Q)\in \mathcal{B}_+$. Let $D$ and $F$ be defined by
\eqref{note:DF}. Then, there exists a constant $c\geq 1$ such that,
for all grid $\varepsilon = r_0 < r_1 < \cdots < r_K = \gamma$, the elastic energy
\begin{align*}
E_{(a,b)} := \int_{\Omega_{(a,b)}}
    \left[ g(F) + h(D)\right] r \mathrm{d} r \mathrm{d} \theta
    = \int_a^b \int_0^{2\pi} \left[
        \kappa r^{1-p} \cdot \left( 2r^2 F \right)^{p/2}
        + r h \left( D \right) \right]
    \mathrm{d} r \mathrm{d} \theta
\end{align*}
satisfies
\begin{align} \label{energy_bound}
c^{-1} r_i^{1-p} \tau_i \leq E_{(r_{i-1},r_i)} \leq c r_{i-1}^{1-p} \tau_i,
\quad 1\leq i \leq K.
\end{align}
\end{lem}
\begin{proof}
It follows from the convexity of $h(\cdot)$ and the hypothesis (H1) that
$$
0 < h(D) \leq \max \left\{ h(c_D^{-1}), h(c_D) \right\}
\triangleq c_h, \quad
\text{on}~ \overline{\Omega'}.
$$
Since $r_{i-1} < \gamma \leq 1$ and $1< p <2$, the hypothesis (H2) implies
\begin{align*}
E_{(r_{i-1},r_i)} &\leq 2\pi \kappa c_F^{p/2}
    \int_{r_{i-1}}^{r_i} r^{1-p} \mathrm{d} r
    + 2\pi c_h
    \int_{r_{i-1}}^{r_i} r \mathrm{d} r \\
&\leq 2\pi \kappa c_F^{p/2} \cdot r_{i-1}^{1-p} \tau_i
    + 2\pi c_h \cdot r_{i-1} \tau_i \\
&\leq 2\pi \left( \kappa c_F^{p/2} + c_h \right)
    \cdot r_{i-1}^{1-p} \tau_i, \\
E_{(r_{i-1},r_i)} &\geq 2\pi \kappa c_F^{-p/2}
    \int_{r_{i-1}}^{r_i} r^{1-p} \mathrm{d} r
    \geq 2\pi \kappa c_F^{-p/2} \cdot r_i^{1-p} \tau_i.
\end{align*}
Hence, the conclusion \eqref{energy_bound} follows by taking
$$
c = \max \left\{ 1, 2\pi \left( \kappa c_F^{p/2} + c_h \right),
    \left( 2\pi \kappa c_F^{-p/2} \right)^{-1} \right\}.
$$
\end{proof}

\begin{lem}
\label{lem:grid}
Let $K \gg \varepsilon^{-1}$ be a sufficiently large integer. Let
$\varepsilon = r_0 < r_1 < \cdots < r_K = \gamma$ be a given grid satisfying
\begin{align}\label{energy_equal}
2r_{K-1}>r_{K} \;\;\;\; \text{and} \;\;\;\;
r_{i-1}^{1-p} \tau_i = r_{i}^{1-p} \tau_{i+1}, \quad 1\leq i \leq K-1,
\end{align}
where $\tau_i := r_i - r_{i-1}$, $1\leq i\leq K$. Then,
we have
\begin{equation}
\label{ineq:rtau}
\sum_{i=1}^{K} \frac{1}{\tau_i} \left( \frac{r_i}{r_{i-1}} \right)^{2p-2}
< 2^{2p} \gamma^{-1} K^2,
\end{equation}
and the energy is radially quasi-equi-distributed on the grid, {\it i.e.}
\begin{equation}
\label{ineq:KE}
2^{1-p} c^{-2}<\frac{ K \cdot E_{(r_{i-1},r_i)} }{ E_{(\varepsilon,\gamma)} } < 2^{p-1} c^2,
\quad \forall i=1,\dotsc,K,
\end{equation}
where $c$ is the same constant in Lemma~\ref{lem:energy}.
\end{lem}
\begin{proof}
By \eqref{energy_equal}, we have
\begin{align}
\label{eq:upsilon}
\frac{\tau_{i+1}}{\tau_i} = \left( \frac{r_i}{r_{i-1}} \right)^{p-1}
    = \left( 1+ \frac{\tau_i}{r_{i-1}} \right)^{p-1}
    = (1+ \Upsilon_i)^{p-1} , \quad 1\leq i \leq K-1,
\end{align}
where denote $\Upsilon_i := \frac{\tau_i}{r_{i-1}} >0$, $1\leq i\leq K-1$.
Notice that, by definition,
\begin{equation}\label{e:Upsilon>}
\Upsilon_{i+1} = \frac{\tau_{i+1}}{r_{i}}
    = \frac{\tau_i (1+ \Upsilon_i)^{p-1}}{r_{i-1}+ \tau_i}
    = \frac{ \frac{\tau_i}{r_{i-1}} (1+ \Upsilon_i)^{p-1}}{1+ \frac{\tau_i}{r_{i-1}}}
    = \Upsilon_i (1+ \Upsilon_i)^{p-2} < \Upsilon_i,
\end{equation}
{\it i.e.} $\Upsilon_i$ is a strictly deceasing function of $i$. On the other hand,
$\frac{\tau_{i+1}}{\tau_i}=(1+ \Upsilon_i)^{p-1} >1$ implies that $\tau_i$ is a strictly
increasing function of $i$, and as a consequence, we have $K \Upsilon_1 \varepsilon =
K \tau_1 < \sum_{i=1}^K \tau_i = \gamma - \varepsilon < K \tau_K$, which yields, since
$K \gg \varepsilon^{-1}$,
\begin{equation}\label{e:Upsilon_1}
\Upsilon_1 < \frac{\gamma-\varepsilon}{K\varepsilon} \ll 1, \;\; \text{and} \;\;
\tau_K > \frac{\gamma-\varepsilon}{K}.
\end{equation}
By \eqref{eq:upsilon} and \eqref{e:Upsilon>},
we also have, for all $1\leq i\leq K-1$,
\begin{align*}
\frac{1}{\tau_i} = \frac{1}{\tau_{i+1}} (1+ \Upsilon_i)^{p-1}
    < \frac{1}{\tau_{i+1}} (1+ \Upsilon_{1})^{p-1} < \cdots
    < \frac{1}{\tau_{K}} (1+ \Upsilon_{1})^{(p-1)(K-i)}.
\end{align*}

Now, express the left-hand side of \eqref{ineq:rtau} as
\begin{equation}
\label{eq:taur}
\sum_{i=1}^{K} \frac{1}{\tau_i} \left( \frac{r_i}{r_{i-1}} \right)^{2p-2}
    = \sum_{i=1}^{K-1} \frac{1}{\tau_i} (1+ \Upsilon_i)^{2p-2}
    + \frac{1}{\tau_K} \left( \frac{r_K}{r_{K-1}} \right)^{2p-2}.
\end{equation}
For the first term on the right hand side of \eqref{eq:taur}, by
\eqref{eq:upsilon}, \eqref{e:Upsilon>} and \eqref{e:Upsilon_1}, we have
\begin{align*}
\sum_{i=1}^{K-1} \frac{1}{\tau_i} (1+ \Upsilon_i)^{2p-2}
&< (1+ \Upsilon_1)^{2p-2} \sum_{i=1}^{K-1} \frac{1}{\tau_i}
    < (1+ \Upsilon_1)^{2p-2} \frac{1}{\tau_K}
        \sum_{i=1}^{K-1} (1+ \Upsilon_{1})^{(p-1)(K-i)} \\
&= (1+ \Upsilon_1)^{2p-2} \frac{1}{\tau_K} \left[
        \frac{ 1- (1+ \Upsilon_1)^{(p-1)K} }{ 1- (1+ \Upsilon_1)^{p-1} } -1\right]
    < \frac{2^{2p-2} (2K-1)}{\tau_K} .
\end{align*}
Since $r_K/r_{K-1}<2$ (see \eqref{energy_equal}), by \eqref{e:Upsilon_1} and
\eqref{eq:taur}, this yields \eqref{ineq:rtau}.

Next, it follows from \eqref{energy_bound}, \eqref{energy_equal} and
$\Upsilon_i \ll 1$, $\forall~i$ that
\begin{align*}
K \cdot E_{(r_{j-1},r_j)}
& \leq c K \cdot r_{j-1}^{1-p} \tau_j
    = c \sum_{i=1}^K r_{i-1}^{1-p} \tau_i
    = c \sum_{i=1}^K r_{i}^{1-p} \tau_i \cdot \left( \frac{r_{i-1}}{r_i} \right)^{1-p} \\
&= c \sum_{i=1}^K r_{i}^{1-p} \tau_i \cdot (1+ \Upsilon_i)^{p-1}
    < 2^{p-1} c \sum_{i=1}^K r_{i}^{1-p} \tau_i \\
&\leq 2^{p-1} c^2 \sum_{i=1}^K E_{(r_{i-1},r_i)}
    = 2^{p-1} c^2 E_{(\varepsilon,\gamma)}, \quad \forall j=1,\dotsc,K.
\end{align*}
This proves the second inequality of \eqref{ineq:KE}.

Notice that, by \eqref{e:Upsilon>} and \eqref{e:Upsilon_1}, we have
$$
\frac{r_{i-1}^{1-p}}{r_{i}^{1-p}}=(1+\Upsilon_i)^{p-1}<2^{p-1}, \quad \forall i=1,\dotsc,K.
$$
Thus, by \eqref{energy_bound} and \eqref{energy_equal},
\begin{equation*}
\label{e:E_subequi}
E_{(r_{i-1},r_i)} < 2^{p-1}c^2 E_{(r_{j-1},r_j)}, \quad \forall i, j=1,\dotsc,K.
\end{equation*}
Denote $E_{max} = \max_{1\le i \le K} E_{(r_{i-1},r_i)}$, then
$2^{p-1}c^2 K E_{(r_{i-1},r_i)} > K E_{max} \ge E_{(\varepsilon,\gamma)}$. This
proves the first inequality of \eqref{ineq:KE}.
\end{proof}

\begin{rem}
For $K$ sufficiently large, it is not difficult to show that there exists
an auxiliary grid $\varepsilon = r_0 < r_1 < \cdots < r_K = \gamma$
such that \eqref{energy_equal} holds.
\end{rem}

The theorem below gives the relative and absolute errors of the elastic energy
$E(P,Q)=E_{(\varepsilon,\gamma)}$ when a cavitation solution $(P,Q)$ is replaced by
its interpolation functions $(I^{NM}P, I^{NM}Q)$.
\begin{thm}
\label{thm:E_tilde}
Let $(P,Q)$ be a cavitation solution satisfying the hypotheses (H1) and (H2).
Then, for $M$, $N$ sufficiently large, there exists a constant $C$ such that
\begin{align*}
\frac{ \left| E(I^{NM}P, I^{NM}Q) - E(P,Q) \right| }{E(P,Q)}
& \leq  C \left( M^{2-l} + N^{1-l} \right) ,\\
\left| E(I^{NM}P, I^{NM}Q) - E(P,Q) \right|
& \leq  C \left( M^{2-l} + N^{1-l} \right) .
\end{align*}
\end{thm}

\begin{proof}
To simplify the notation, we denote (see \eqref{note:DF})
$$
\widetilde{P}:= I^{NM}P,~ \widetilde{Q}:= I^{NM}Q, \quad
\widetilde{D}:= D(\widetilde{P},\widetilde{Q}),~
\widetilde{F}:= F(\widetilde{P},\widetilde{Q}).
$$
By \eqref{note:DFg} and \eqref{energy}, the energy error can be bounded as follows
\begin{align}\label{e:energy_error_bounds}
\left| E(\widetilde{P},\widetilde{Q}) - E(P,Q) \right|
&\leq \int_0^{2\pi} \int_{\varepsilon}^\gamma r \left[
        \left| g(\widetilde{F}) - g(F) \right| +
        \left| h(\widetilde{D}) - h(D) \right| \right]
        \mathrm{d} r \mathrm{d} \theta \nonumber \\
&\leq \int_0^{2\pi} \int_{\varepsilon}^\gamma r^{1-p} \kappa
        \left| \left( 2r^2 \widetilde{F} \right)^{p/2}
             - \left( 2r^2 F \right)^{p/2} \right|
        \mathrm{d} r \mathrm{d} \theta \nonumber \\
&\qquad + \int_0^{2\pi} \int_{\varepsilon}^\gamma
        r \left| h(\widetilde{D}) - h(D) \right|
        \mathrm{d} r \mathrm{d} \theta
    \triangleq I + II.
\end{align}

By the hypotheses (H1), (H2), and as a consequence of Theorem~\ref{thm:PQ}, we have $(P,Q)$, $(\widetilde{P},\widetilde{Q})$ and their first order derivatives are all bounded, and
\begin{subequations}
\label{error_INM}
\begin{align}
\| \widetilde{P} - P \|_{\omega,\Omega'}
    &\leq c \left( M^{-l} + N^{-l} \right) ,&
    \| \widetilde{Q} - Q \|_{\omega,\Omega'}
    &\leq c \left( M^{-l} + N^{-l} \right) ,\\
\| \widetilde{P}_\rho - P_\rho \|_{\omega,\Omega'}
    &\leq c \left( M^{2-l} + N^{-l} \right) ,&
    \| \widetilde{Q}_\rho - Q_\rho \|_{\omega,\Omega'}
    &\leq c \left( M^{2-l} + N^{-l} \right) ,\\
\| \widetilde{P}_\phi - P_\phi \|_{\omega,\Omega'}
    &\leq c \left( M^{-l} + N^{1-l} \right) ,&
    \| \widetilde{Q}_\phi - Q_\phi \|_{\omega,\Omega'}
    &\leq c \left( M^{-l} + N^{1-l} \right) ,
\end{align}
\end{subequations}
where $\| \cdot\|_{\omega,\Omega'}$ is the weighted $L^2$-norm on $\Omega'$.

Thus, by \eqref{note:rDrF} and recalling $\rho_r = 2/(\gamma-\varepsilon)$, we have
\begin{align*}
\left| r\widetilde{D}-rD \right|
&\leq \rho_r \left( \left| \widetilde{P} \widetilde{P}_\rho (\widetilde{Q}_\phi+1)
        - P P_\rho (Q_\phi+1) \right| +
        \left| \widetilde{P} \widetilde{P}_\phi \widetilde{Q}_\rho
        - P P_\phi Q_\rho \right| \right) \\
&\leq c \left(  \left| \widetilde{P} - P \right| +
                \left| \widetilde{P}_\rho - P_\rho \right| +
                \left| \widetilde{P}_\phi - P_\phi \right| +
                \left| \widetilde{Q}_\rho - Q_\rho \right| +
                \left| \widetilde{Q}_\phi - Q_\phi \right| \right) ,\\
\left| 2r^2\widetilde{F} - 2r^2F \right|
&\leq r^2 \rho_r^2 \left| \widetilde{P}_\rho^2 + \widetilde{P}^2 \widetilde{Q}_\rho^2
        - ( P_\rho^2 + P^2 Q_\rho^2 ) \right|
        + \left| \widetilde{P}_\phi^2 - P_\phi^2 \right| \\
&\qquad + \left| \widetilde{P}^2 (\widetilde{Q}_\phi+1)^2 - P^2 (Q_\phi+1)^2 \right| \\
&\leq c \left(  \left| \widetilde{P} - P \right| +
                \left| \widetilde{P}_\rho - P_\rho \right| +
                \left| \widetilde{P}_\phi - P_\phi \right| +
                \left| \widetilde{Q}_\rho - Q_\rho \right| +
                \left| \widetilde{Q}_\phi - Q_\phi \right| \right) ,
\end{align*}
and as a consequence, it follows from \eqref{error_INM} that
\begin{subequations}
\begin{align}
\| r\widetilde{D}-rD \|_{\omega,\Omega'} &\leq c \left(M^{2-l}+N^{1-l} \right),
\label{e:rD_error}   \\ \label{e:2r2F_error}
\| 2r^2\widetilde{F} - 2r^2F \|_{\omega,\Omega'} &\leq c \left(M^{2-l}+N^{1-l} \right).
\end{align}
\end{subequations}

By hypothesis (H1), (H2) and Theorem~\ref{thm:PQ}, both $D>0$ and $\widetilde{D}>0$ are bounded away from $0$ and $+\infty$, hence, by \eqref{e:rD_error}, we have
\begin{align}\label{e:II}
II &= \int_0^{2\pi} \int_{\varepsilon}^\gamma
        r |h'(\vartheta_1)| \left| \widetilde{D} - D \right|
        \mathrm{d} r \mathrm{d} \theta
    \leq c \int_0^{2\pi} \int_{-1}^1
        \left| r \widetilde{D} - r D \right|
        \mathrm{d} \rho \mathrm{d} \phi \nonumber \\
&   \leq c \| r \widetilde{D}-r D \|_{\omega,\Omega'}
    \leq c \left(M^{2-l}+N^{1-l} \right),
\end{align}
where $\vartheta_1$ is between $\widetilde{D}$ and $D$, and thus $h'(\vartheta_1)$ is bounded.

On the other hand, let $\varepsilon = r_0 < r_1 < \cdots < r_K = \gamma$ be an
auxiliary grid in radial direction satisfying the conditions of Lemma~\ref{lem:grid},
then, by \eqref{energy_bound} and \eqref{ineq:KE}, we have
\begin{align}\label{e:I/E_1}
\frac{I}{ E_{(\varepsilon,\gamma)} }
 &< \displaystyle \sum_{i=1}^K \int_{ \Omega_{(r_{i-1},r_i)} } \frac{2^{p-1} c^{2}}{
 K E_{(r_{i-1},r_i)}}\,
 \kappa \left| \left( 2r^2 \widetilde{F} \right)^{p/2}
   - \left( 2r^2 F \right)^{p/2} \right|\, \mathrm{d} r \mathrm{d} \theta
   \nonumber    \\
& \leq 2^{p-1} c^3 \kappa \sum_{i=1}^K \frac{1}{K \tau_i}
               \left( \frac{r_{i-1}}{r_i} \right)^{1-p}
                \int_{ \Omega_{(r_{i-1},r_i)} }
                \left| \left( 2r^2 \widetilde{F} \right)^{p/2}
                    - \left( 2r^2 F \right)^{p/2} \right|
                \mathrm{d} r \mathrm{d} \theta .
\end{align}
Since hypotheses (H1), (H2) and Lemma~\ref{lem:Liguo} implies that both $r^2F$ and
$r^2\widetilde{F}$ are bounded, by the H\"{o}lder inequality, we have
\begin{align*}
&\int_{ \Omega_{(r_{i-1},r_i)} }
        \left| \left( 2r^2 \widetilde{F} \right)^{p/2}
            - \left( 2r^2 F \right)^{p/2} \right|
        \mathrm{d} r \mathrm{d} \theta
    = \int_{ \Omega_{(r_{i-1},r_i)} }
        \frac{p}{2}| \vartheta_2^{p/2-1} |
        \left| 2r^2 \widetilde{F} - 2r^2 F \right|
        \mathrm{d} r \mathrm{d} \theta \\
&\leq c \int_{ \Omega_{(r_{i-1},r_i)} }
        \left| 2r^2 \widetilde{F} - 2r^2 F \right|
        \mathrm{d} r \mathrm{d} \theta
    \leq c \sqrt{\tau_i} \| 2r^2 \widetilde{F}
        - 2r^2 F \|_{L^2( \Omega_{(r_{i-1},r_i)} )} ,
\end{align*}
where $\vartheta_2$ is between $2r^2 \widetilde{F}$ and $2r^2 F$, and thus
$|\vartheta_2^{p/2-1}|$ is also bounded. Substituting this into \eqref{e:I/E_1}
and applying the H\"{o}lder inequality, we have,
by \eqref{ineq:rtau} and \eqref{e:2r2F_error},
\begin{align}\label{e:I/E_2}
\frac{I}{ E_{(\varepsilon,\gamma)} }
&\leq c \sum_{i=1}^K \frac{1}{K \sqrt{\tau_i} }
               \left( \frac{r_i}{r_{i-1}} \right)^{p-1} \cdot
               \| 2r^2 \widetilde{F}
                    - 2r^2 F \|_{L^2( \Omega_{(r_{i-1},r_i)} )} \nonumber \\
&\leq c \left[ \sum_{i=1}^K \frac{1}{K^2} \cdot \frac{1}{\tau_i}
               \left( \frac{r_i}{r_{i-1}} \right)^{2p-2} \right]^{1/2} \cdot
        \left[ \sum_{i=1}^K
                \| 2r^2 \widetilde{F}
                    - 2r^2 F \|_{L^2( \Omega_{(r_{i-1},r_i)} )}^2
                \right]^{1/2} \nonumber  \\
&\leq c \| 2r^2 \widetilde{F} - 2r^2 F \|_{L^2( \Omega_{(\varepsilon,\gamma)} )}
    \leq c \| 2r^2 \widetilde{F} - 2r^2 F \|_{\omega,\Omega'}
    \leq c \left(M^{2-l}+N^{1-l} \right) .
\end{align}
The proof is completed by combining\eqref{e:energy_error_bounds}, \eqref{e:II} and
\eqref{e:I/E_2}.
\end{proof}

Notice that $\mathcal{B}_+^{NM} \subset \mathcal{B}_+$, the result of Theorem~\ref{thm:E_tilde}
allows us to obtain the following elastic energy error estimate for the discrete cavitation solution.
\begin{thm}
\label{thm:E_NM}
Let a cavitation solution $(P,Q)$ satisfy the hypotheses (H1), (H2) and be a
global energy minimizer of $E(\cdot,\cdot)$ in $\mathcal{B}_+$. Let $(P^{NM},Q^{NM})$
be a global energy minimizer of $E(\cdot,\cdot)$ in $\mathcal{B}_+^{NM}$. Then,
for $M$, $N$ sufficiently large, there exists a constant $C>0$ such that
\begin{equation}
\label{e:E^NM}
E(P,Q) \le E(P^{NM},Q^{NM}) \le E(P,Q) + C \left( M^{2-l} + N^{1-l}\right).
\end{equation}
\end{thm}
\begin{proof}
The first inequality is a direct consequence of $\mathcal{B}_+^{NM} \subset \mathcal{B}_+$.
By \eqref{det_error} and for $M$, $N$ sufficiently large, we have $(I^{NM}P,I^{NM}Q) \in \mathcal{B}_+^{NM}$. Then, the second inequality follows from Theorem~\ref{thm:E_tilde} and
$E(P^{NM},Q^{NM}) \le E(I^{NM}P, I^{NM}Q)$.
\end{proof}

Let $(P,Q)\in \mathcal{B}_+$ be a global energy minimizer of
$E(\cdot,\cdot)$ in $\mathcal{B}_+$, let $(I^{NM}P,I^{NM}Q)\in \mathcal{B}_+^{NM}$
be its interpolation functions. Let $(P^{NM},Q^{NM})\in \mathcal{B}_+^{NM}$ be
global energy minimizers of $E(\cdot,\cdot)$ in $\mathcal{B}_+^{NM}$. Denote
$\mathbf{\bar{u}}$, $\mathbf{\bar{u}}^{NM}$ and $\mathbf{U}^{NM}$ as the corresponding
functions on $\Omega_{(\varepsilon,\gamma)}$ defined by $(P,Q)$,
$(I^{NM}P,I^{NM}Q)$ and $(P^{NM},Q^{NM})$ respectively via the coordinates transformations
\eqref{coordinate:polar} and \eqref{coordinate:rho_phi}.
Then, \eqref{e:E^NM} can be rewritten as
\begin{equation}\label{e:E(u)NM}
E(\mathbf{\bar{u}}) = \inf_{\mathbf{v}\in \mathcal{A}_{\varepsilon}} E(\mathbf{v}) \le
E(\mathbf{U}^{NM})  \le E(\mathbf{\bar{u}}^{NM}) \le
E(\mathbf{\bar{u}}) + C\left( M^{2-l} + N^{1-l}\right).
\end{equation}

According to Theorem~4.9 in \cite{Su2015} and its proof, for a conforming discrete
approximation method of the cavitation problem, the inequality \eqref{e:E(u)NM} implies
the convergence of the discrete cavitation solutions. Thus, we have the following
convergence theorem. For the convenience of the readers, we sketch its proof below.

\begin{thm}
Let a cavitation solution $(P,Q)$ satisfy the hypotheses (H1), (H2) and be a global
energy minimizer of $E(\cdot,\cdot)$ in $\mathcal{B}_+$. Let $(P^{NM},Q^{NM})$
be global energy minimizers of $E(\cdot,\cdot)$ in $\mathcal{B}_+^{NM}$.
Let $\mathbf{U}^{NM}$ correspond to $(P^{NM},Q^{NM})$ under the transformations
\eqref{coordinate:polar} and \eqref{coordinate:rho_phi}. Then, there exists a subsequence, still
denoted as $\{\mathbf{U}^{NM}\}$, and a function $\mathbf{u} \in
\mathcal{A}_{(\varepsilon,\gamma)}( \mathbf{u}_0 )$,
such that $\mathbf{U}^{NM} \rightarrow \mathbf{u}$ in
$W^{1,p}(\Omega_{(\varepsilon,\gamma)})$ and $\mathbf{u}$ is a global
energy minimizer of $E(\cdot)$ in $\mathcal{A}_{(\varepsilon,\gamma)}( \mathbf{u}_0 )$.
\end{thm}

\begin{proof}
Since $h>0$ is a convex function satisfying the growth conditions
\eqref{note:h}, $1<p<2$ and $\mathbf{U}^{NM}$ satisfies the Direchlet boundary condition,
by \eqref{e:E(u)NM} and the De La Vall\'{e}e Poussin theorem \cite{Poussin}, we conclude that
$\{|\nabla \mathbf{U}^{NM}|^p\}_{N,M \rightarrow \infty}$ and $\{\det \nabla
\mathbf{U}^{NM}\}_{N,M \rightarrow \infty}$ are equi-integrable. As a consequence,
there exists a subsequence, still denoted as $\{\mathbf{U}^{NM}\}$, a function
$\mathbf{u} \in W^{1,p}(\Omega_{(\varepsilon,\gamma)})$ and a function
$\zeta \in L^1(\Omega_{(\varepsilon,\gamma)})$ such that
\begin{equation}\label{e:Uweakconv}
\mathbf{U}^{NM} \rightharpoonup \mathbf{u} ~\text{in}~
    W^{1,p}(\Omega_{(\varepsilon,\gamma)}), \quad
    \mathbf{U}^{NM} \rightarrow \mathbf{u} ~a.e., \quad
\det \nabla \mathbf{U}^{NM} \rightharpoonup \zeta  ~\text{in}~
    L^1(\Omega_{(\varepsilon,\gamma)}).
\end{equation}

Hence, by $\det \nabla \mathbf{U}^{NM} >0$, $a.e.$, we have $\zeta \geq 0$, $a.e.$.
We conclude that $\zeta >0$, $a.e.$. Suppose otherwise, {\it i.e.}
$\zeta =0$ on a set $S$ with positive measure, then there exists a subsequence,
still denoted as $\{\mathbf{U}^{NM}\}$, such that $\int_S |\det \nabla \mathbf{U}^{NM}|
\mathrm{d} \mathbf{x} \rightarrow 0$ and $\det \nabla \mathbf{U}^{NM} \rightarrow 0$,
$a.e.$ on the set $S$, which, by \eqref{note:h}, implies
$h(\det \nabla \mathbf{U}^{NM}) \rightarrow \infty$, $a.e.$ on the set $S$.
Thus, by the Fatou lemma, we have and $E(\mathbf{U}^{NM}) \rightarrow \infty$, which
contradicts to $\displaystyle \varlimsup_{N,M\rightarrow \infty} E(\mathbf{U}^{NM}) <\infty$.

Thanks to Theorem~3 in \cite{Henao2010} and Theorem~3 in \cite{Henao2011}, as a
consequence of \eqref{e:Uweakconv}, $\zeta >0$, $a.e.$ and the continuity of
$\mathbf{U}^{NM}$, we have $\zeta = \det \nabla \mathbf{u}$, $a.e.$ and
$\mathbf{u}$ is one-to-one a.e.. In addition, it is easily verified that
$\mathbf{u}|_{\partial \Omega} = \mathbf{u}_0$. Hence
$\mathbf{u} \in \mathcal{A}_{(\varepsilon,\gamma)}( \mathbf{u}_0 )$. On the other hand, since
$\displaystyle E(\mathbf{u}) \leq \varliminf_{N,M\rightarrow \infty} E(\mathbf{U}^{NM})$,
due to the weakly lower semi-continuity of $E(\cdot)$ on
$W^{1,p}(\Omega_{(\varepsilon,\gamma)})$ (see the theorem 5.4 in \cite{Ball1981}),
we conclude from $\mathbf{u} \in \mathcal{A}_{(\varepsilon,\gamma)}( \mathbf{u}_0 )$
and \eqref{e:E(u)NM} that
$\mathbf{u}$ is a global minimizer of $E(\cdot)$ in
$\mathcal{A}_{(\varepsilon,\gamma)}( \mathbf{u}_0 )$ and
$\displaystyle E(\mathbf{u}) = \lim_{N,M\rightarrow \infty} E(\mathbf{U}^{NM})$.

Since $h$ is a convex function, it follows from $\det \nabla \mathbf{U}^{NM} \rightharpoonup
\det \nabla \mathbf{u}$ in $L^1(\Omega_{(\varepsilon,\gamma)})$ that
\begin{align*}
&E(\mathbf{u}) - \kappa \int_{\Omega_{(\varepsilon,\gamma)}}
    |\nabla \mathbf{u}|^p \mathrm{d} \mathbf{x}
= \int_{\Omega_{(\varepsilon,\gamma)}} h(\det \nabla \mathbf{u})
    \mathrm{d} \mathbf{x}
\leq \varliminf_{N,M\rightarrow \infty} \int_{\Omega_{(\varepsilon,\gamma)}}
    h(\det \nabla \mathbf{U}^{NM}) \mathrm{d} \mathbf{x} \\
&\quad
= \varliminf_{N,M\rightarrow \infty} \left( E(\mathbf{U}^{NM}) -
    \kappa \int_{\Omega_{(\varepsilon,\gamma)}} |\nabla \mathbf{U}^{NM}|^p
    \mathrm{d} \mathbf{x} \right)
= E(\mathbf{u}) - \kappa \varlimsup_{N,M\rightarrow \infty}
    \int_{\Omega_{(\varepsilon,\gamma)}} |\nabla \mathbf{U}^{NM}|^p
    \mathrm{d} \mathbf{x},
\end{align*}
which leads to $\displaystyle \varlimsup_{N,M\rightarrow \infty}
\norm{\mathbf{U}^{NM}}_{W^{1,p}(\Omega_{(\varepsilon,\gamma)})} \leq
\norm{\mathbf{u}}_{W^{1,p}(\Omega_{(\varepsilon,\gamma)})}$.
Thus, by the uniform convexity of $W^{1,p}(\Omega_{(\varepsilon,\gamma)})$
(see \cite{Adams1975}) and $\mathbf{U}^{NM} \rightharpoonup \mathbf{u}$ in
$W^{1,p}(\Omega_{(\varepsilon,\gamma)})$, it follows from proposition 3.30
in \cite{Brezis1983} that $\mathbf{U}^{NM} \rightarrow \mathbf{u}$ in
$W^{1,p}(\Omega_{(\varepsilon,\gamma)})$. This completes the proof.
\end{proof}

\section{Numerical Experiments and Results}

In our numerical experiments, the stored energy density function $W(\cdot)$
is taken of the form \eqref{note:W} with
\begin{align}
p=\frac32, \quad \kappa =\frac23, \quad
h(t)= 2^{-1/4} \left( \frac{(t-1)^2}{2}+\frac{1}{t} \right).
\end{align}
The reference configuration is $\Omega_{(\varepsilon,\gamma)} =
\mathbb{B}_\gamma(\mathbf{0}) \setminus \overline{\mathbb{B}_\varepsilon(\mathbf{0})}$,
$(0<\varepsilon \ll \gamma \leq 1)$. We consider $\mathbf{u}_0(\mathbf{x})=
\lambda \mathbf{x},~ \mathbf{x} \in \partial \mathbb{B}_\gamma( \mathbf{0})$,
$\lambda >1$ in the radially-symmetric case and $\mathbf{u}_0(\mathbf{x}) =
\left[ \lambda_1 x_1, \lambda_2 x_2 \right]^T$, $\mathbf{x} \in \partial
\mathbb{B}_\gamma( \mathbf{0})$, $\lambda_1,\lambda_2 >1$ in the non-radially-symmetric case.

By \eqref{det_error} and the hypotheses (H1), (H2), we expect to have
$(I^{NM}P,I^{NM}Q) \in \mathcal{B}_+^{NM}$ for sufficiently large $N$ and $M$.
Before proceeding to the numerical experiments, we first check in
Table~\ref{tab:incompress_det} the orientation preservation condition
$D(I^{NM}P,I^{NM}Q)>0$ for the exact cavitation solution $(P,Q)$ in the
radially-symmetric case for incompressible elastic materials,
since in such a case the cavity solutions have a simple explicit form
$R(r)=\sqrt{\lambda^2+r^2-\gamma^2}$ in the polar coordinate systems.
By $D(I^{NM}P,I^{NM}Q) = \frac{\rho_r}{r} \cdot I^{NM} P \cdot (I^{NM} P)_\rho$,
we only need to check whether $(I^{NM} P)_\rho >0$ is satisfied. In fact, whenever
$(I^{NM} P)_\rho >0$ is satisfied, we have $D(I^{NM}P,I^{NM}Q)\approx 1$.
The data shown in Table~\ref{tab:incompress_det} suggest that the orientation
preservation condition $D>0$ should not impose much real additional restrictions on
the choice of $N$ and $M$ in practical computations.

\begin{table}[H]
\centering
\footnotesize
\renewcommand{\arraystretch}{1.0}{
\caption{\footnotesize The minima of $(I^{NM} P)_\rho$ in $\Omega'$ in various cases (independent of $N$) }
\begin{tabular*}{\textwidth}{@{\extracolsep{\fill}}@{~~}ccrrrrrr}
\toprule 
    \raisebox{-2.00ex}[0cm][0cm]{ $\lambda \times \gamma$ } &
    \raisebox{-2.00ex}[0cm][0cm]{ $(\varepsilon,\gamma)$ } &
    \multicolumn{6}{c}{ $M$ } \\
    \cline{3-8}
    & & 4~~~ & 6~~~ & 8~~~ & 10~~~ & 12~~~ & 14~~~ \\
\midrule 
\multirow{3}{*}{ 2 }  & $(10^{-2},1)$ & 2.66e-3 & 2.86e-3
        & 2.86e-3 & 2.86e-3 & 2.86e-3 & 2.86e-3  \\[-1mm]
    & $(10^{-3},1) $  & 8.57e-5  & 2.92e-4
        & 2.88e-4 & 2.88e-4 & 2.88e-4 & 2.88e-4  \\[-1mm]
    & $(10^{-4},1) $  & -1.75e-4 & 3.27e-5
        & 2.88e-5 & 2.89e-5 & 2.89e-5 & 2.89e-5  \\[1mm]
\multirow{3}{*}{1.25}  & $(10^{-3},10^{-1})$ & 3.97e-5 & 3.97e-5
        & 3.97e-5 & 3.97e-5 & 3.97e-5 & 3.97e-5  \\[-1mm]
    & $(10^{-4},10^{-2})$   & 3.96e-7 & 3.96e-7
        & 3.96e-7 & 3.96e-7 & 3.96e-7 & 3.96e-7  \\[-1mm]
    & $(10^{-5},10^{-3}) $  & 3.96e-9 & 3.96e-9
        & 3.96e-9 & 3.96e-9 & 3.96e-9 & 3.96e-9  \\
\bottomrule 
\end{tabular*}
\label{tab:incompress_det} }
\end{table}

Next, we investigate the effect of the number of quadrature points $N',M'$ used in
\eqref{EL:final_D}. Figure \ref{fig:numer_integ} shows the convergence behavior of the
errors on the cavity radius for various $N'/N$ (with $M=32$ and $M'=8M$ fixed) and
$M'/M$ (with $N=16$ and $N'=2N$ fixed), where for the
non-radially-symmetric case $\Omega_{(\varepsilon,\gamma)}=\Omega_{(10^{-4},1)}$
with $\lambda_1 = 2.4$ and $\lambda_2 = 2$,
and for the radially-symmetric case $\Omega_{(\varepsilon,\gamma)}=\Omega_{(10^{-2},1)}$
with $\lambda = 2$. To balance the accuracy and computational cost, we set in our numerical
experiments below $N' = 2N$ and $M' = 8M$.

\begin{figure}[H]
\centering
\subfigure[$M=32$, non-symmetric]{
\includegraphics[width=4.8cm,height=5.4cm]{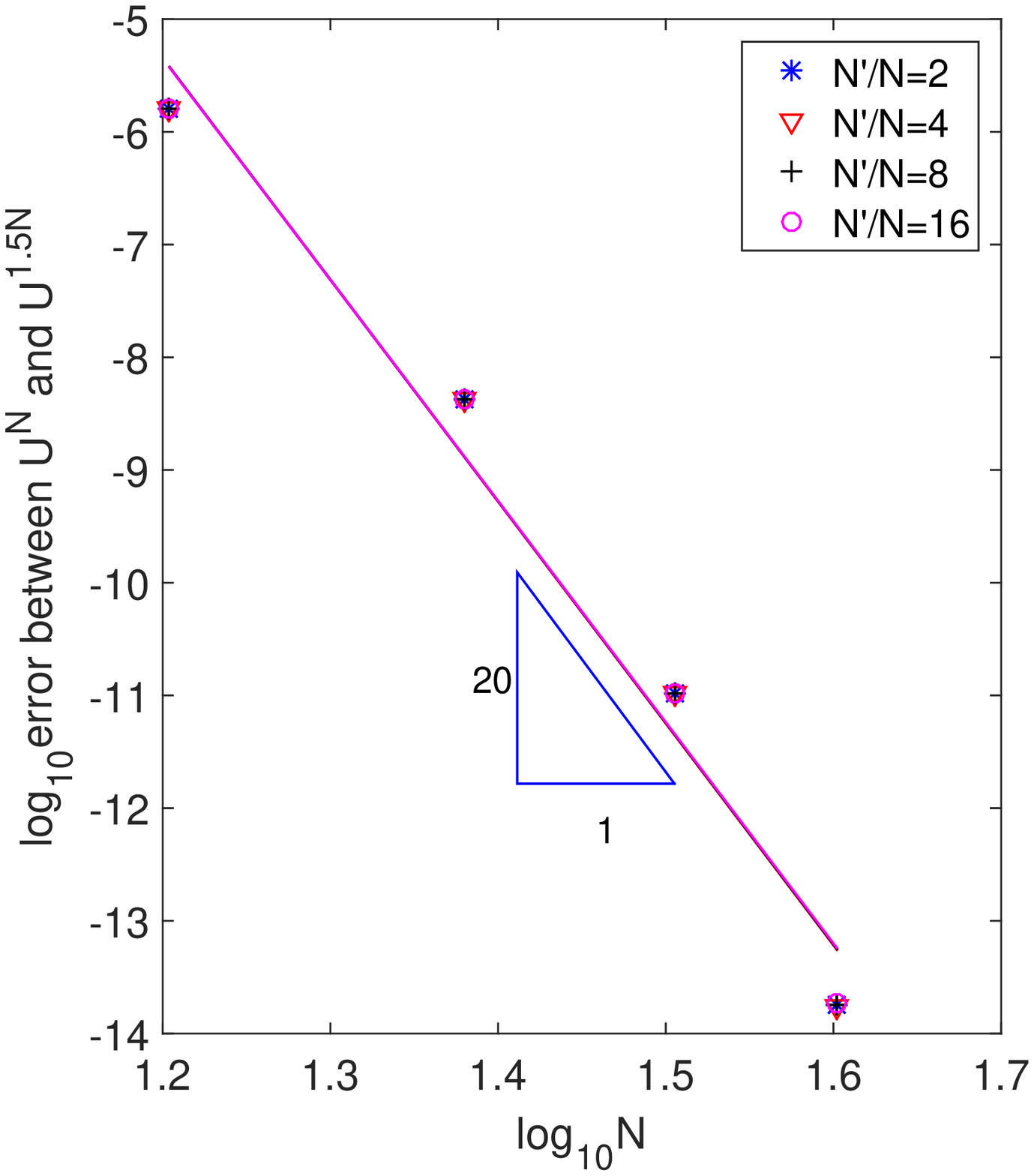}
}
\subfigure[$N=16$, non-symmetric]{
\includegraphics[width=4.8cm,height=5.4cm]{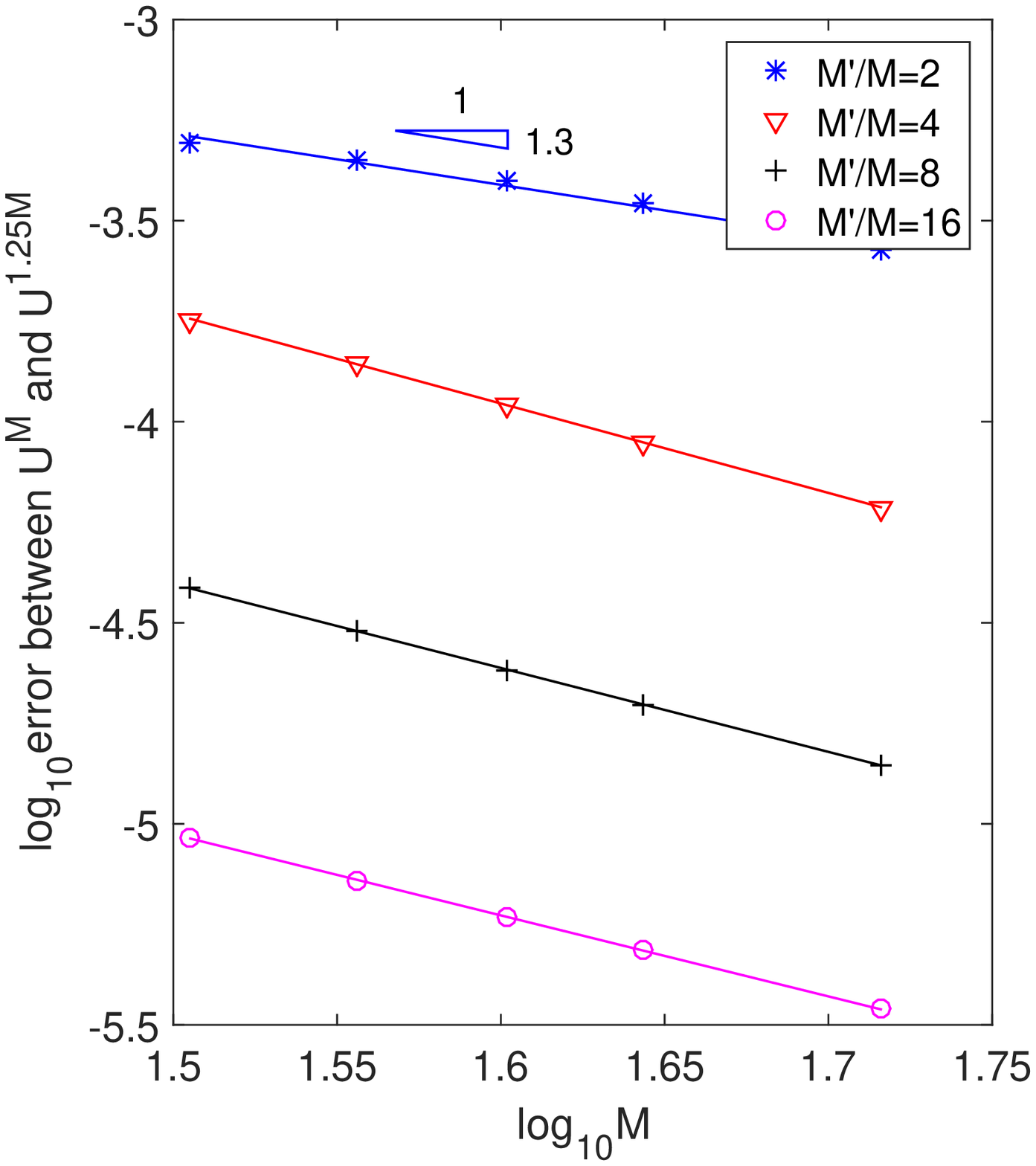}
}
\subfigure[$N=16$, symmetric]{
\includegraphics[width=4.8cm,height=5.4cm]{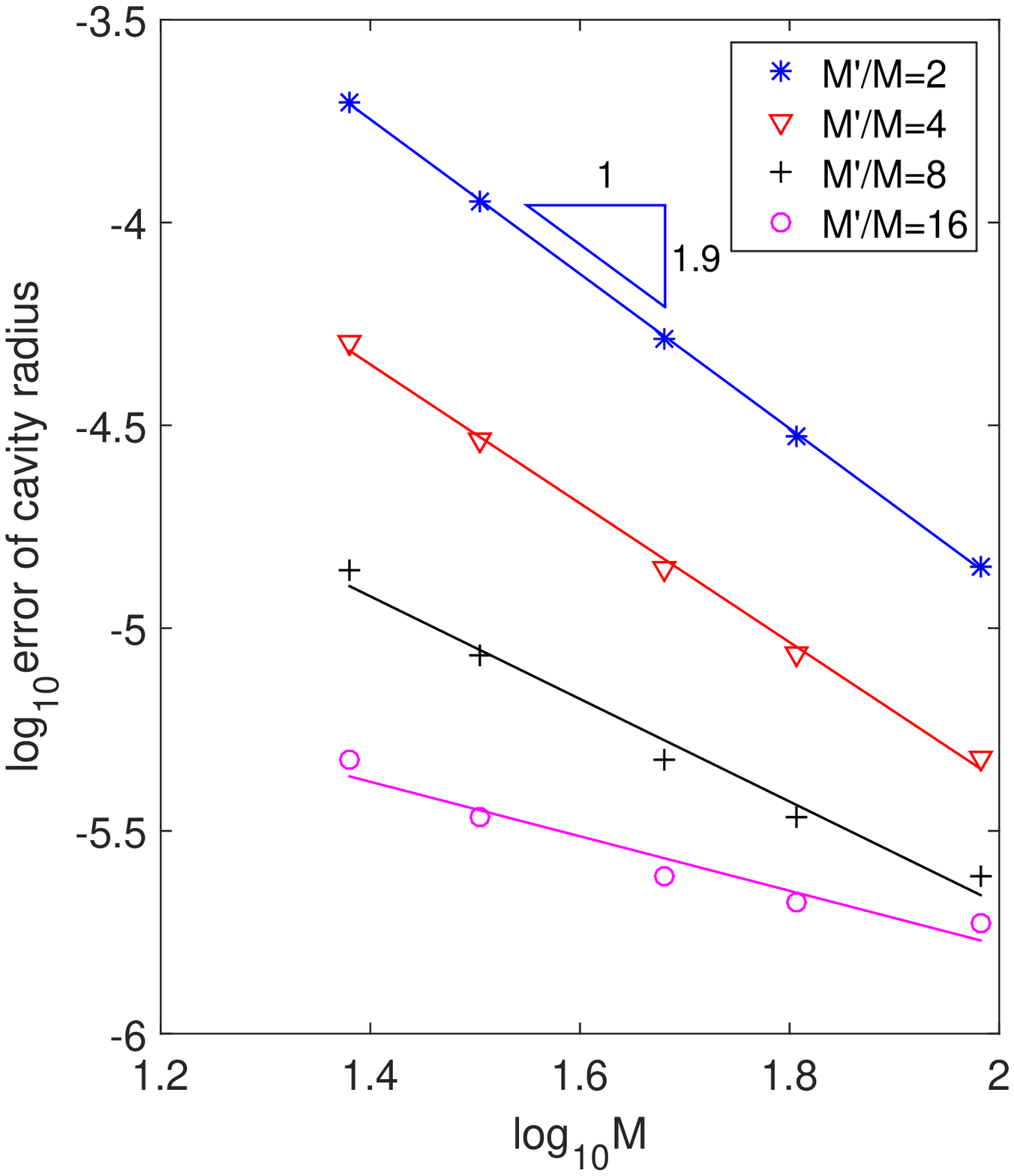}
}
\caption{Effect of $N'/N$, $M'/M$ on the cavity radius errors.}
\label{fig:numer_integ}
\end{figure}

\subsection{Radially-Symmetric Case}

In the radially-symmetric case with $\mathbf{u}_0(\mathbf{x})=\lambda \mathbf{x}$, $\lambda >1$,
the cavitation solution $\mathbf{u}$ can be written in polar coordinates systems as
\begin{equation*}
R = s(r),~~ \Theta=\theta, \quad
\forall~ (r,\theta) \in [\varepsilon,\gamma]\times [0,2\pi],
\end{equation*}
where $s(r)$ satisfies $s(\varepsilon)>0$ and $s(\gamma) = \lambda \cdot \gamma$. Since in theory the
numerical solution is independent of the circumferential DOF (degree of freedom) $N$,
we fix $N=16$ and examine the effect of the radial DOF $M$ on the numerical performance
of our method. In comparison, high precision numerical solutions
to the equivalent 1-dimensional ODE boundary value problems \cite{Sivaloganathan2009},
obtained by the {\it ode15s} routine in MATLAB with the tolerance $10^{-16}$, are
taken as the exact solutions.

For the standard case of $\gamma=1$, $\lambda = 2$, and
$\varepsilon= 10^{-3}$, $10^{-4}$,
the convergence behavior of our numerical cavitation solutions $\mathbf{U}^{NM}$
is shown in Figure~\ref{fig:error_ode_3}, \ref{fig:error_ode_4},
where $L^2_{\omega}$ and $W^{1,p}$ represent
$L^2_{\omega}(\Omega')$-norm and
$W^{1,p}(\Omega_{(\varepsilon,\gamma)})$-semi-norm respectively.

For $\gamma=1$ and $\varepsilon = 10^{-2},10^{-3},10^{-4}$, we show in
Figure \ref{fig:lambda_crit_symmetry} the numerical results obtained with $M=32$ on the
$\lambda$-$R^{NM}(\varepsilon)$ ({\it i.e.} the expansion rate on the outer boundary
against the cavity radius on the inner boundary) graph.
Figure~\ref{fig:lambda_crit_sivaloganathan} shows, for $\varepsilon=10^{-4}$,
the convergence of our numerical
results to that of the 1-dimensional ODE solution.

\begin{figure}[H]
\centering
\subfigure[$\varepsilon = 10^{-3}$, $N=16$.]{
\includegraphics[width=7.2cm,height=5.4cm]{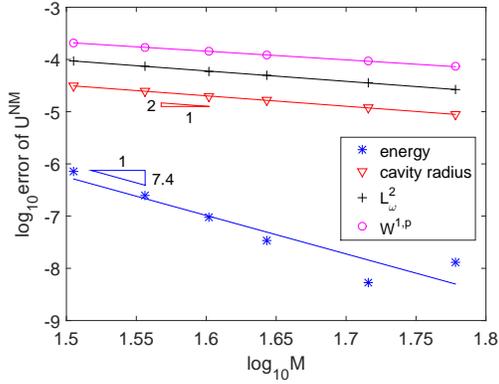}
\label{fig:error_ode_3}
}
\subfigure[$\varepsilon = 10^{-4}$, $N=16$.]{
\includegraphics[width=7.2cm,height=5.4cm]{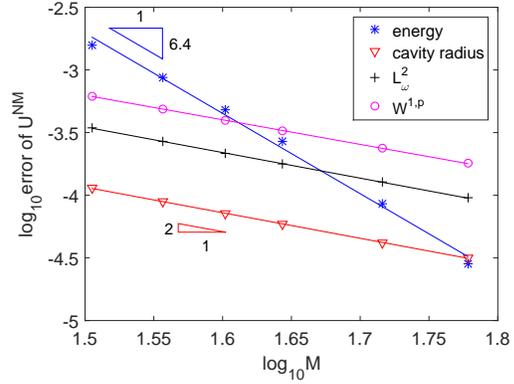}
\label{fig:error_ode_4}
}
\caption{The convergence behavior of radially-symmetric $\mathbf{U}^{NM}$ with $N=16$ fixed.}
\end{figure}

\begin{figure}[H]
\begin{minipage}{.5\linewidth}
\centering
\subfigure{ \includegraphics[width=7.2cm,height=5.4cm]{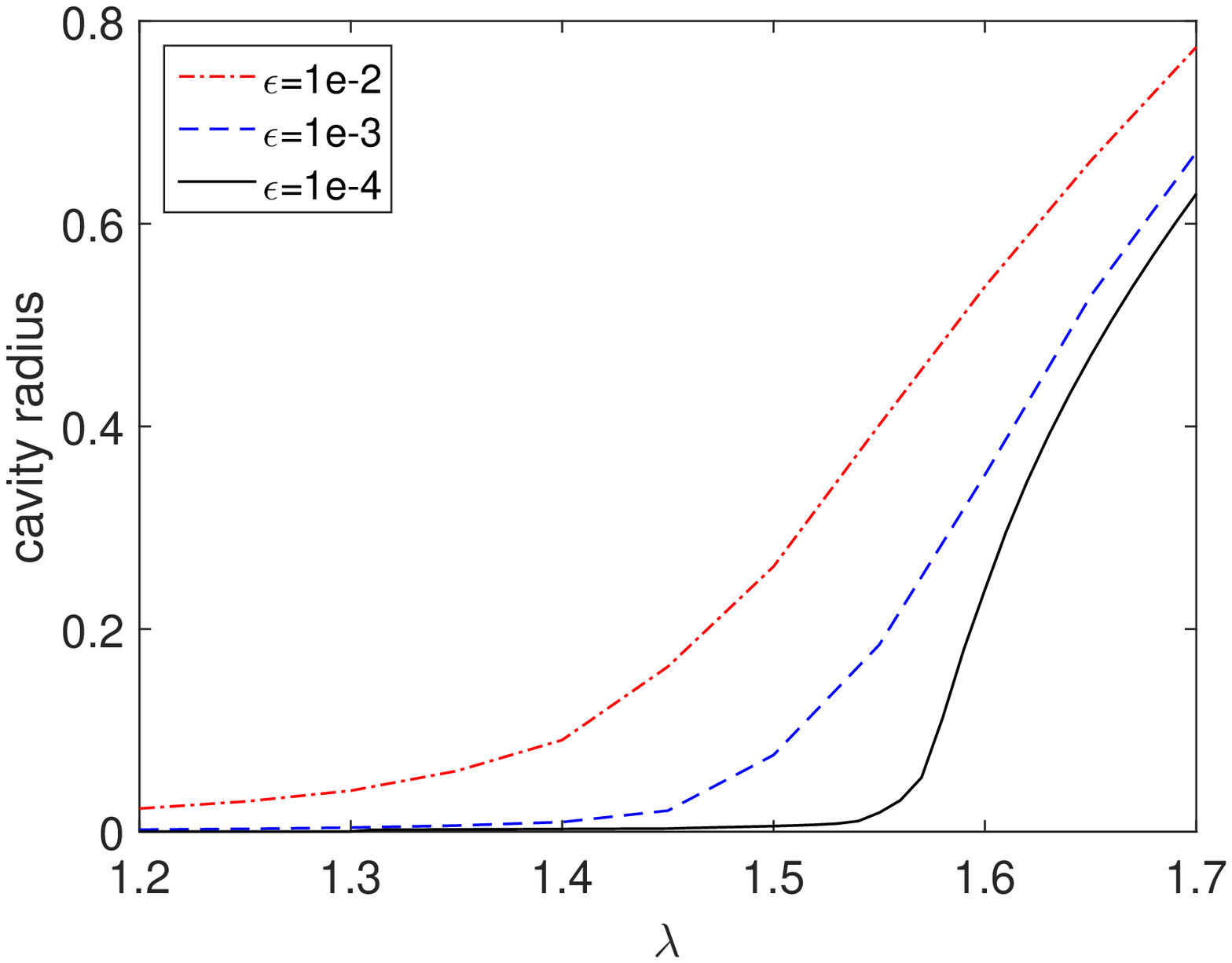} }
\caption{$\lambda$-$R^{NM}(\varepsilon)$, $N=16$, $M=32$.}
\label{fig:lambda_crit_symmetry}
\end{minipage}
\begin{minipage}{.1\linewidth} \end{minipage}
\begin{minipage}{.5\linewidth}
\centering
\subfigure{ \includegraphics[width=7.2cm,height=5.4cm]{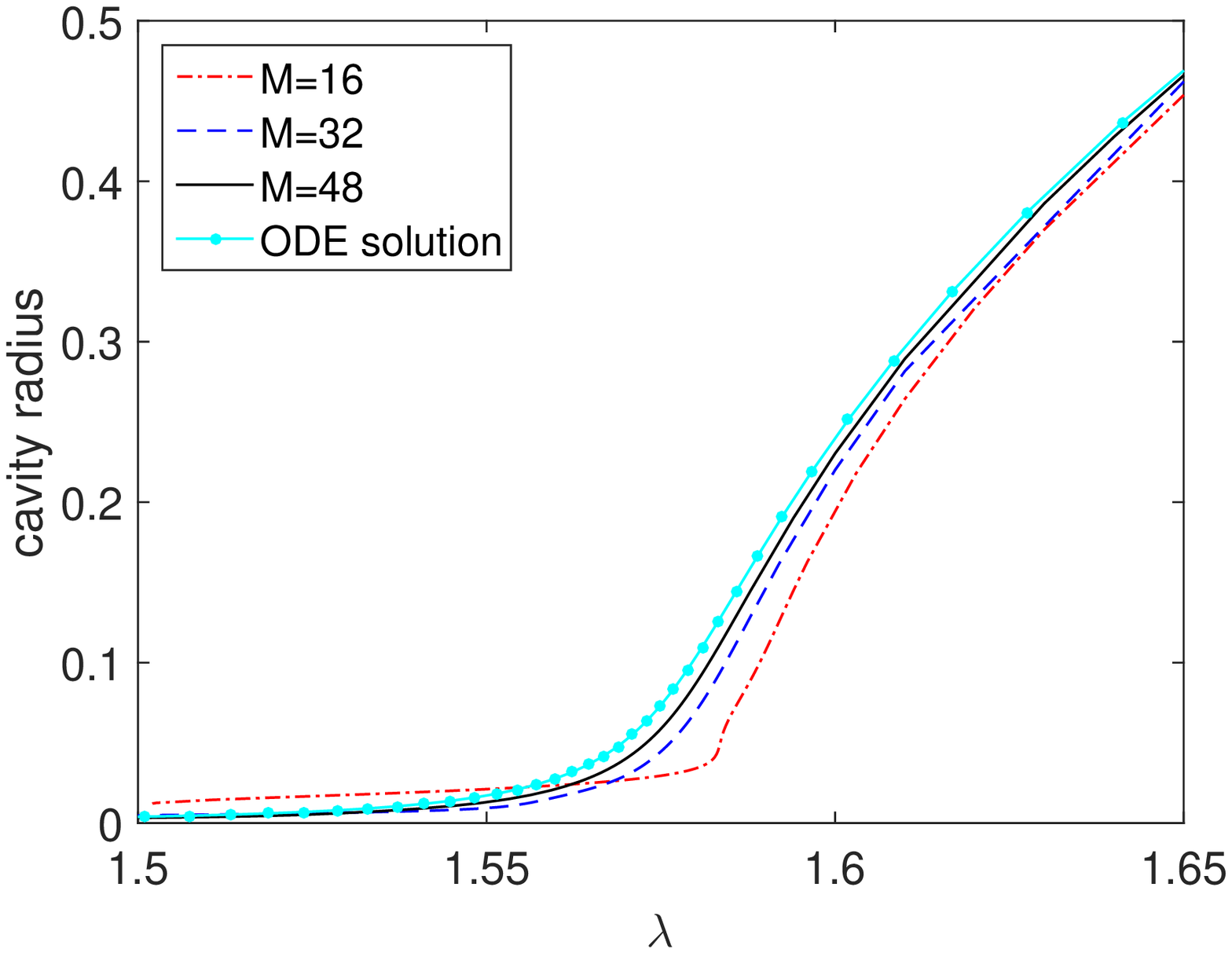} }
\caption{$\lambda$-$R^{NM}(10^{-4})$, $N=16$.}
\label{fig:lambda_crit_sivaloganathan}
\end{minipage}
\end{figure}

\begin{figure}[H]
\centering
\subfigure[radially-symmetric, energy error]{
\includegraphics[width=7.2cm,height=5.4cm]{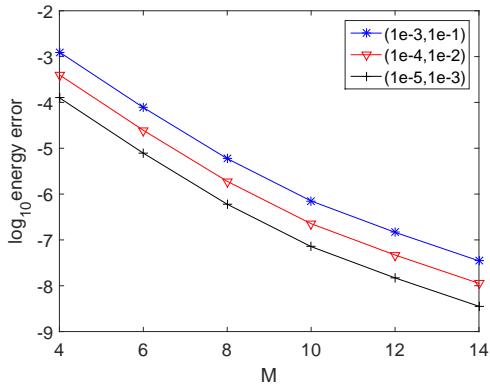}
}
\subfigure[radially-symmetric, cavity radius error]{
\includegraphics[width=7.2cm,height=5.4cm]{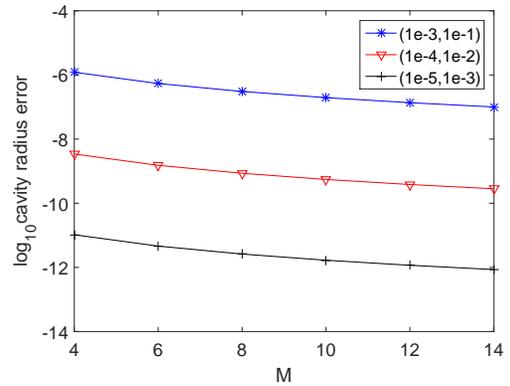}
}
\caption{The convergence behavior on $\Omega_{(\varepsilon,\gamma)}$ with small $\gamma$.}
\label{fig:subdomain}
\end{figure}

To explore the potential of the method in coupling with a domain decomposition method,
especially when combining with a finite element method in a multi-defects problem,
we examine the convergence behavior of our algorithm on a small neighbourhood of the defect.
Taking $\Omega_{(10^{-3},10^{-1})}$, $\Omega_{(10^{-4},10^{-2})}$ and
$\Omega_{(10^{-5},10^{-3})}$ as the reference
configurations and setting $\lambda \cdot \gamma=1.25$, we show in Figure~\ref{fig:subdomain}
the energy error and cavity radius error as a function of $M$ (with $N=16$ fixed),
where it is clearly seen that
high precision numerical results can be obtained with rather small $M$.

\subsection{Non-radially Symmetric Case}

For the non-radially symmetric case, we consider the circular ring reference configuration
$\Omega_{(\varepsilon,\gamma)}$ with oval boundary stretch
$\mathbf{u}_0 (\mathbf{x}) = \left[ \lambda_1 x_1, \lambda_2 x_2 \right]^T$,
$\lambda_1$, $\lambda_2>1$.
Assuming that the error of the numerical solution $\mathbf{U}^{NM}$ satisfies
\begin{align}
\label{eq:error}
q^{NM} \approx q^{\infty} + c_1 N^{-\nu_1} + c_2 M^{-\nu_2},
\end{align}
where $q^{\infty}$ and $q^{NM}$ represent the exact and numerical results of a
specific quantity, such as the elastic energy, semi-major axis and semi-minor axis etc.,
and $c_1$, $c_2$, $v_1$, $v_2$ are the corresponding parameters to be determined
by the least squares data fitting.

For $\Omega_{(\varepsilon,\gamma)}=\Omega_{(10^{-4},1)}$, $\lambda_1 = 2.4$ and
$\lambda_2 = 2$, we show in Figure~\ref{fig:nonsymmetry_N} the errors between
$\mathbf{U}^N$ and $\mathbf{U}^{1.5N}$ with $M=32$ fixed, and in
Figure~\ref{fig:nonsymmetry_M} the errors between
$\mathbf{U}^M$ and $\mathbf{U}^{1.25 M}$ with $N=16$ fixed, where
$L^2_{\omega}$ and $W^{1,p}$ represent $L^2_{\omega}(\Omega')$-norm and
$W^{1,p}(\Omega_{(\varepsilon,\gamma)})$-semi-norm respectively. The regressed quantities
and parameters are shown in Table~\ref{tab:fit}.

\begin{figure}[H]
\centering
\subfigure[non-radially-symmetric, $M=32$]{
\includegraphics[width=7.2cm,height=5.4cm]{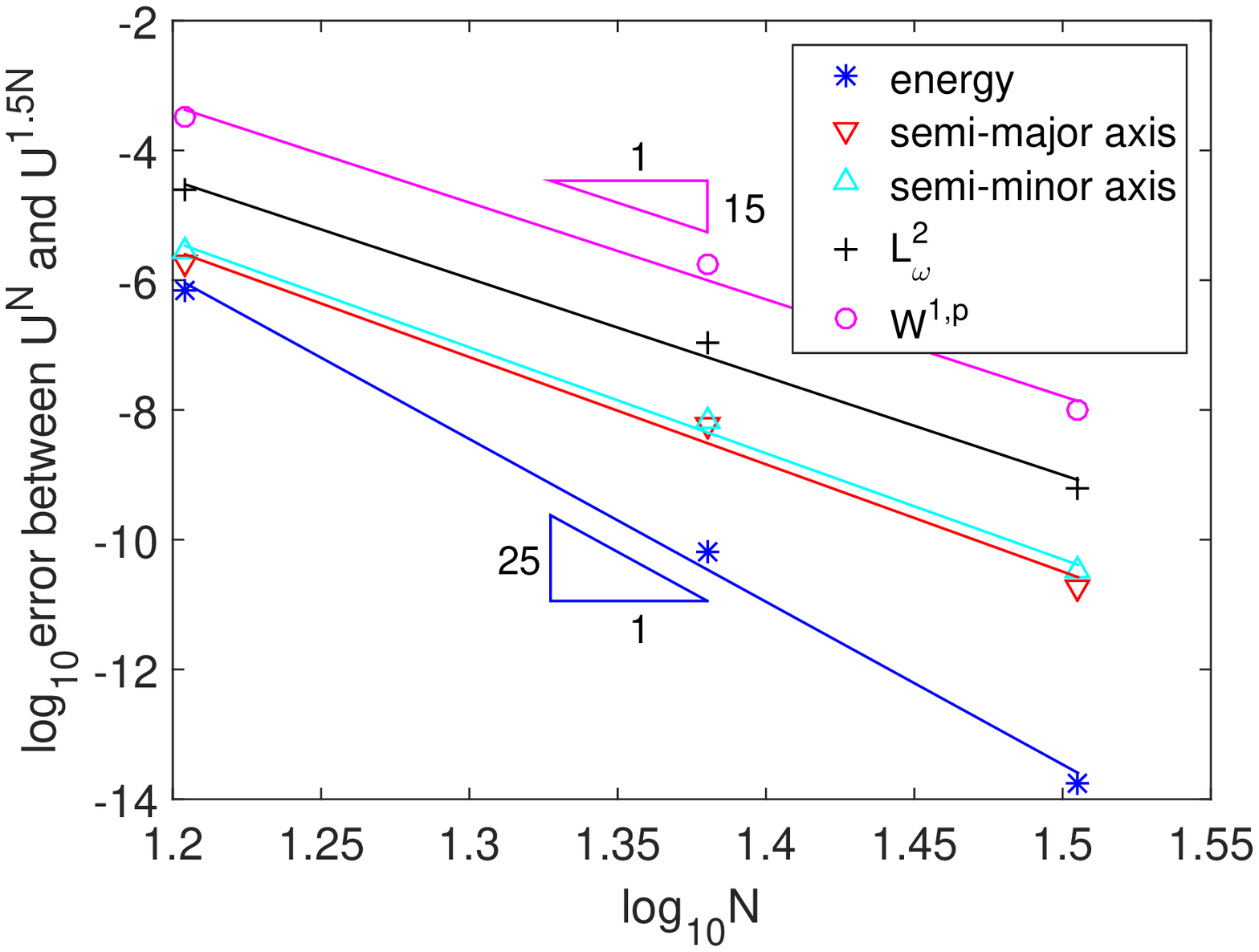}
\label{fig:nonsymmetry_N}
}
\subfigure[non-radially-symmetric, $N=16$]{
\includegraphics[width=7.2cm,height=5.4cm]{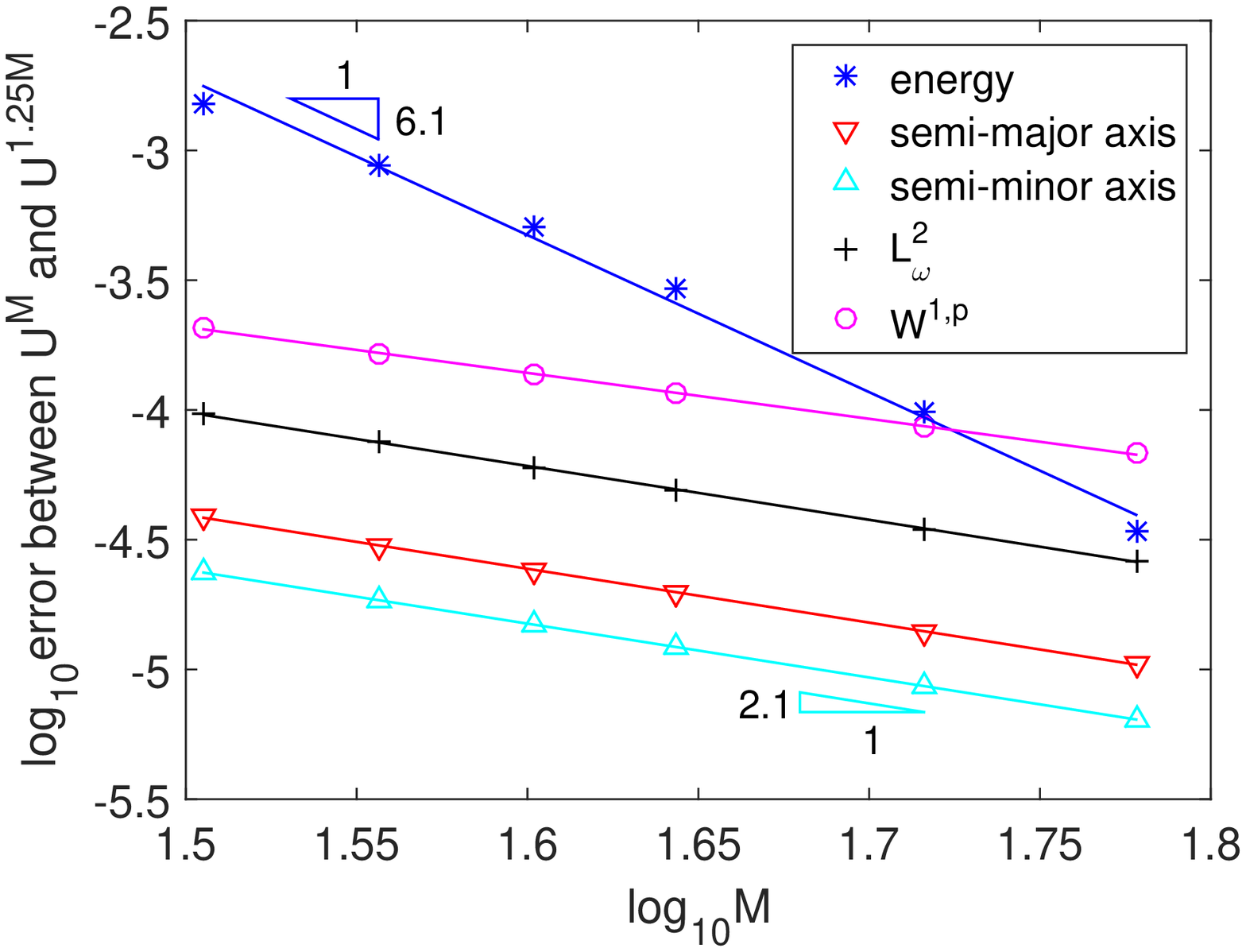}
\label{fig:nonsymmetry_M}
}
\caption{The convergence behavior of  the non-radially-symmetric $\mathbf{U}^{NM}$.}
\end{figure}

\begin{table}[H]
\centering
\footnotesize
\renewcommand{\arraystretch}{1.0}{
\caption{The regressed quantities and parameters for the non-radially-symmetric case.}
\begin{tabular*}{\textwidth}{@{\extracolsep{\fill}}@{~~}crrrrr}
\toprule 
$q$ & $c_1$ & $c_2$ & $\nu_1$ & $\nu_2$ & $q^{\infty}$ \\
\midrule 
energy          &  1.45e+24 & -3.10e+6 & 25 & 6.1 & 22.85959048 \\[-1mm]
semi-major axis & -2.02e+14 & -1.37e-1 & 17 & 2.1 &  1.67481624 \\[-1mm]
semi-minor axis & -1.61e+14 & -8.42e-2 & 16 & 2.1 &  1.42872097 \\
\bottomrule 
\end{tabular*}
\label{tab:fit} }
\end{table}

As a comparison, we show in Figure~\ref{fig:error_twoM} the corresponding errors obtained
in the same way for the radially-symmetric case with $\lambda=2$, and show in Table~\ref{tab:fit_s}
the regressed quantities and parameters. It is clearly seen that the regressed formula \eqref{eq:error}
produces quite sharp numerical results in the radially-symmetric case.

\begin{figure}[H]
\centering
\subfigure[symmetric, $\varepsilon = 10^{-3}$, $N=16$]{
\includegraphics[width=7.2cm,height=5.4cm]{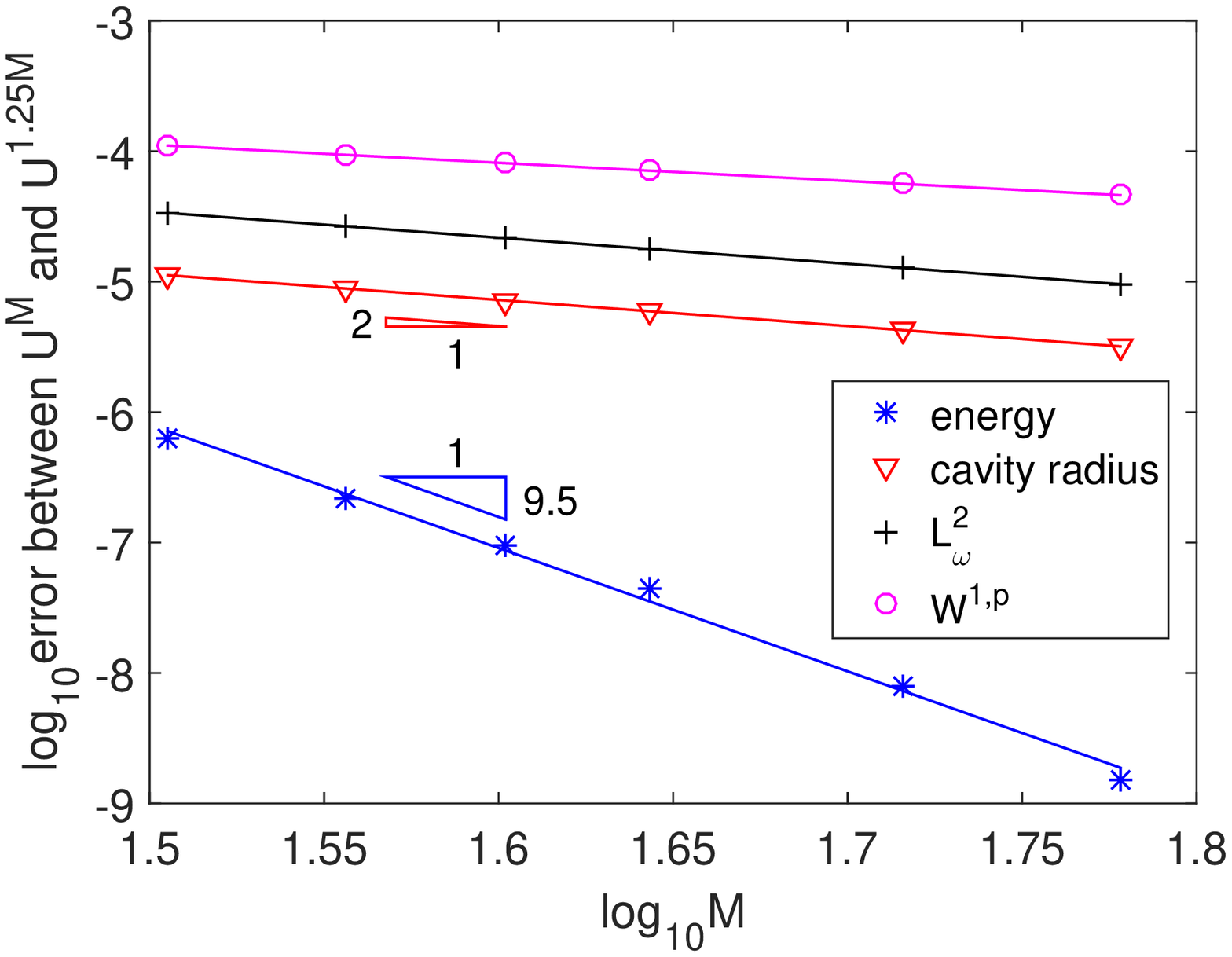}
}
\subfigure[symmetric, $\varepsilon = 10^{-4}$, $N=16$]{
\includegraphics[width=7.2cm,height=5.4cm]{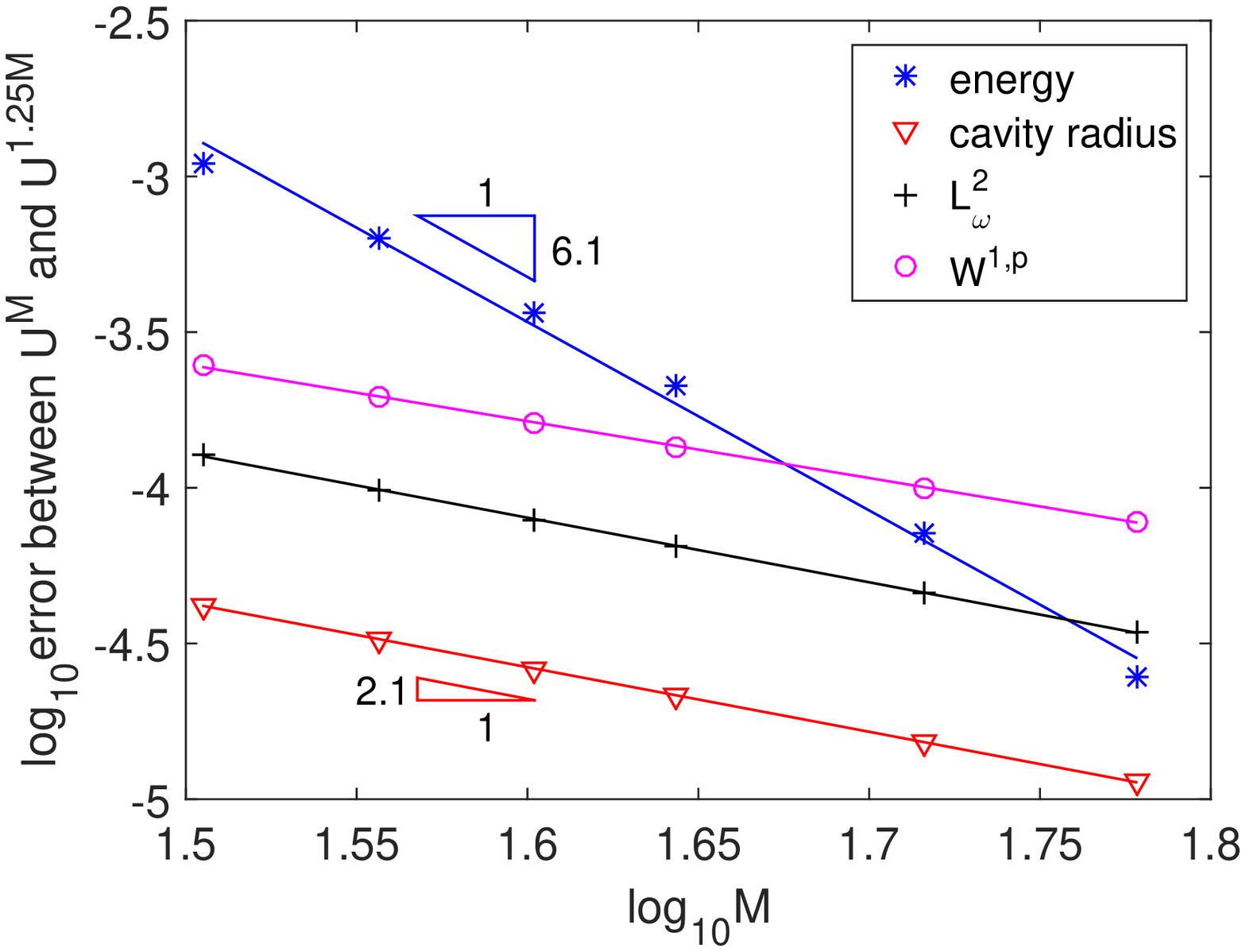}
}
\caption{The convergence behavior of  the radially-symmetric $\mathbf{U}^{NM}$.}
\label{fig:error_twoM}
\end{figure}

\begin{table}[H]
\centering
\footnotesize
\renewcommand{\arraystretch}{1.0}{
\caption{The regressed quantities and parameters for the radially-symmetric case.}
\begin{tabular*}{\textwidth}{@{\extracolsep{\fill}}@{~~}ccrrrr}
\toprule
$\varepsilon$ & $q$ & $c_2$ & $\nu_2$ & $q^{\infty}$ & ODE solution \\
\midrule
\multirow{2}{*}{ $10^{-3}$ } & energy &  -1.38e+8 & 9.5 & 18.61960091 & 18.61960090 \\[-1mm]
                      & cavity radius & -3.18e-2  & 2.0 &  1.26772534 &  1.26772534 \\[-1mm]
\multirow{2}{*}{ $10^{-4}$ } & energy &  -2.22e+6 & 6.1 & 18.87582778 & 18.87582146 \\[-1mm]
                      & cavity radius & -1.49e-1  & 2.1 &  1.25228561 &  1.25228643 \\[-1mm]
\bottomrule
\end{tabular*}
\label{tab:fit_s} }
\end{table}

To see how well the regressed formula \eqref{eq:error} fits the data, we show in
Figure~\ref{fig:fit_error_radius} the errors on the cavity dimensions, {\it i.e.}
the cavity radius in the radially-symmetric case and the cavity major and minor axes
in the non-radially-symmetric case, and in \ref{fig:fit_error_energy} the errors on
the elastic energy, between the corresponding quantities produced by the numerical
solution $\mathbf{U}^{NM}$ and the regressed formula \eqref{eq:error}
respectively. In particular, compare also to Figure~\ref{fig:error_ode_4}, it is clearly
seen that the regressed formula \eqref{eq:error} is highly accurate and reliable.

\begin{figure}[H]
\centering
\subfigure[errors on cavity dimensions]{
\includegraphics[width=7.2cm,height=5.4cm]{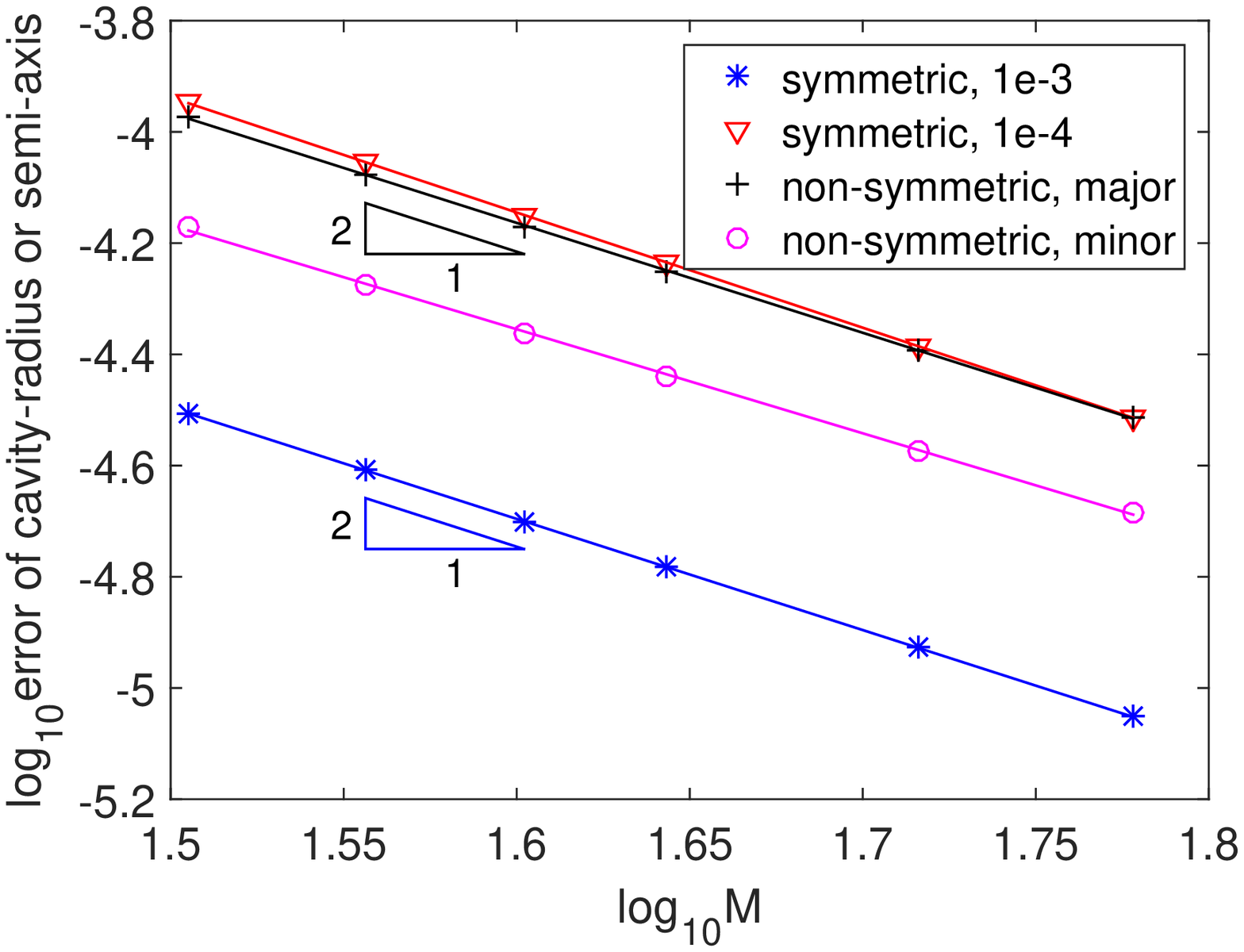}
\label{fig:fit_error_radius}
}
\subfigure[errors on on the elastic energy]{
\includegraphics[width=7.2cm,height=5.4cm]{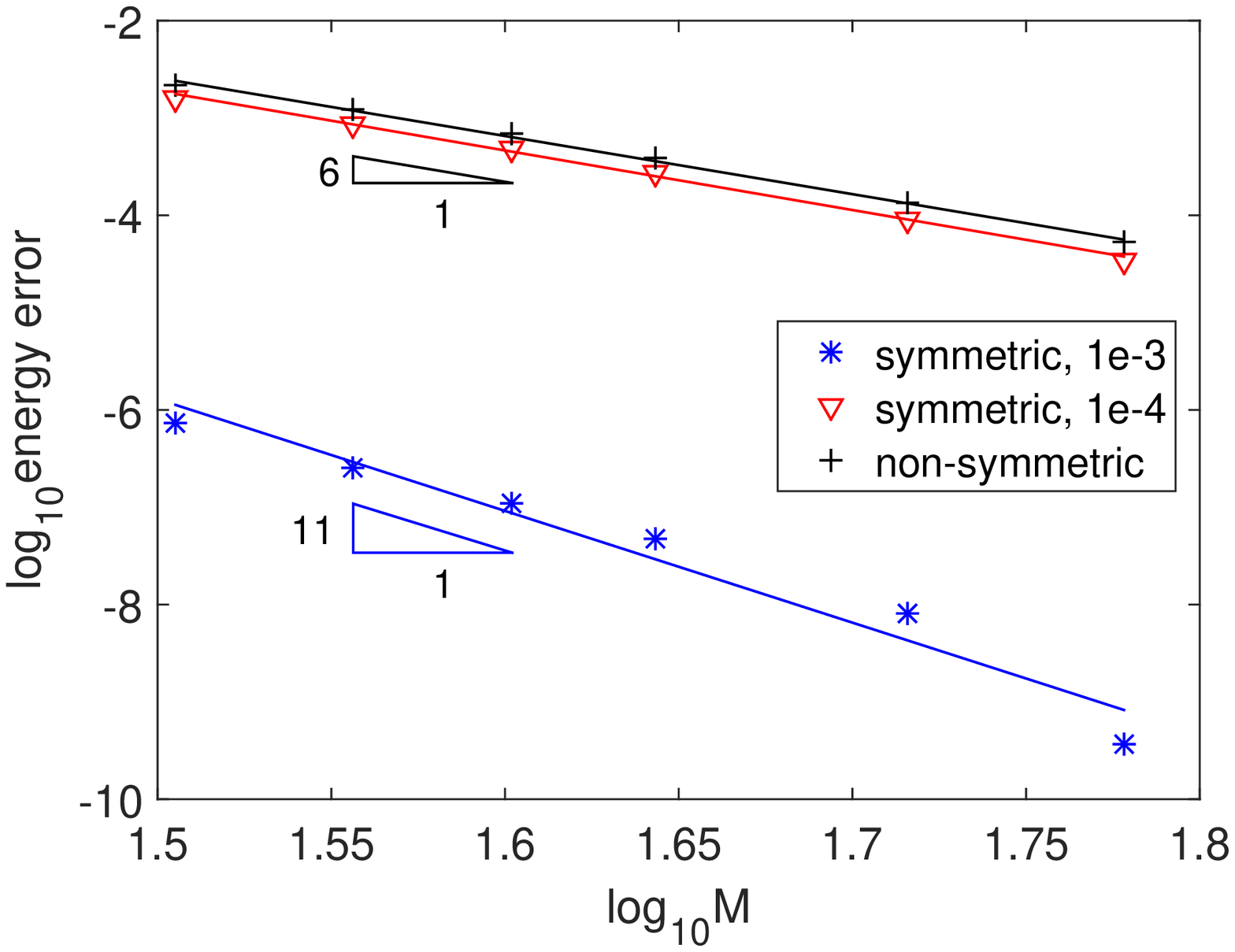}
\label{fig:fit_error_energy}
}
\caption{Errors on the key quantities produced by $\mathbf{U}^{NM}$ and the regressed data.}
\label{fig:fit_error}
\end{figure}

To examine how the axial ratio of the oval stretch affect the critical displacement,
we show in Figure~\ref{fig:lambda_crit_nonsymmetry} the semi-major and semi-minor axes
of the numerical cavity formed as functions of $\lambda_1$ for
$\lambda_1 / \lambda_2 = 1.2,1.3,1.4$, with
$\Omega_{(\varepsilon,\gamma)}=\Omega_{(10^{-4},1)}$, $N=16$ and $M=32$, where
it is obviously seen that both are monotonously increasing functions.

\begin{figure}[H]
\centering
\subfigure[semi-major axis]{
\includegraphics[width=7.2cm,height=5.4cm]{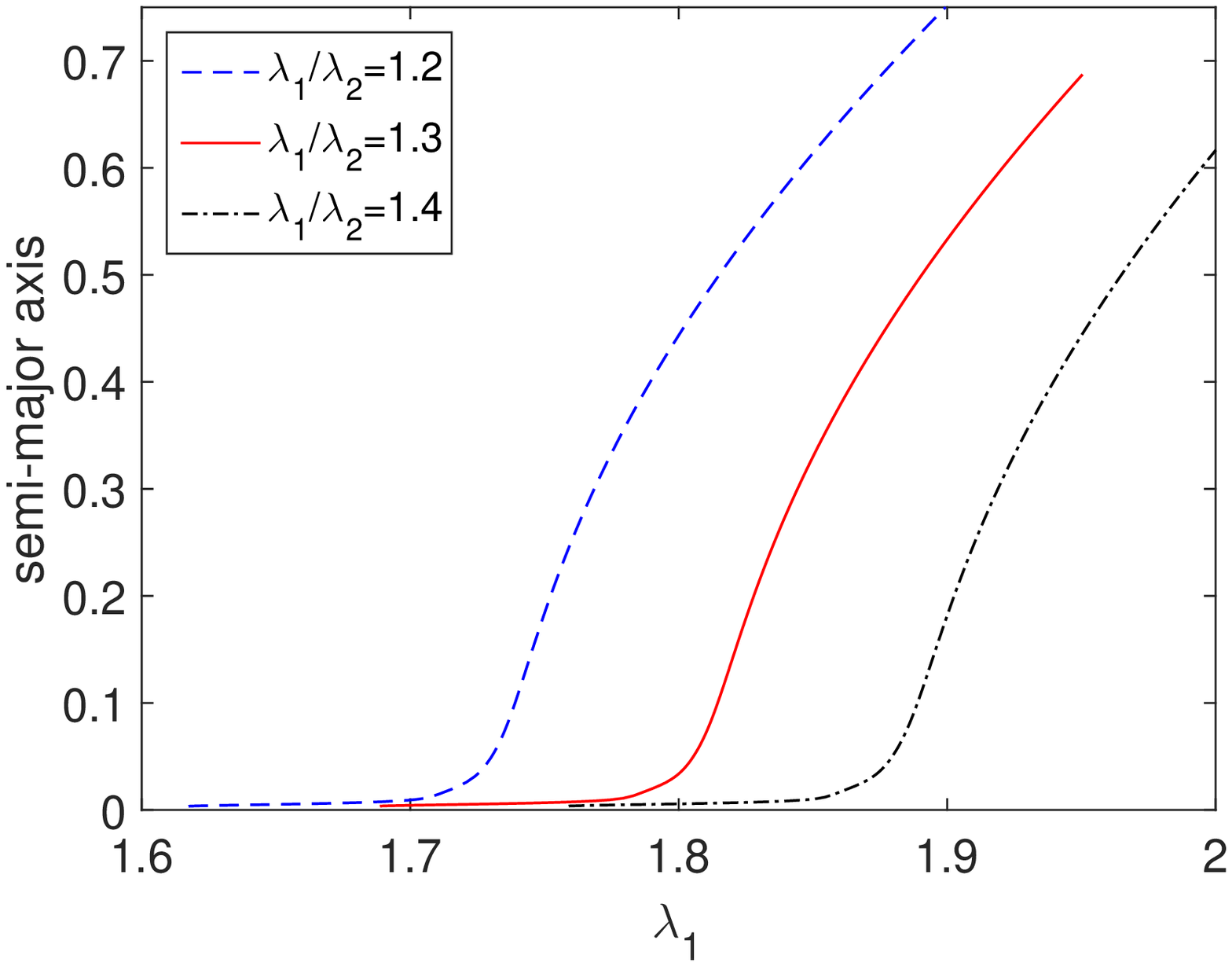}
}
\subfigure[semi-minor axis]{
\includegraphics[width=7.2cm,height=5.4cm]{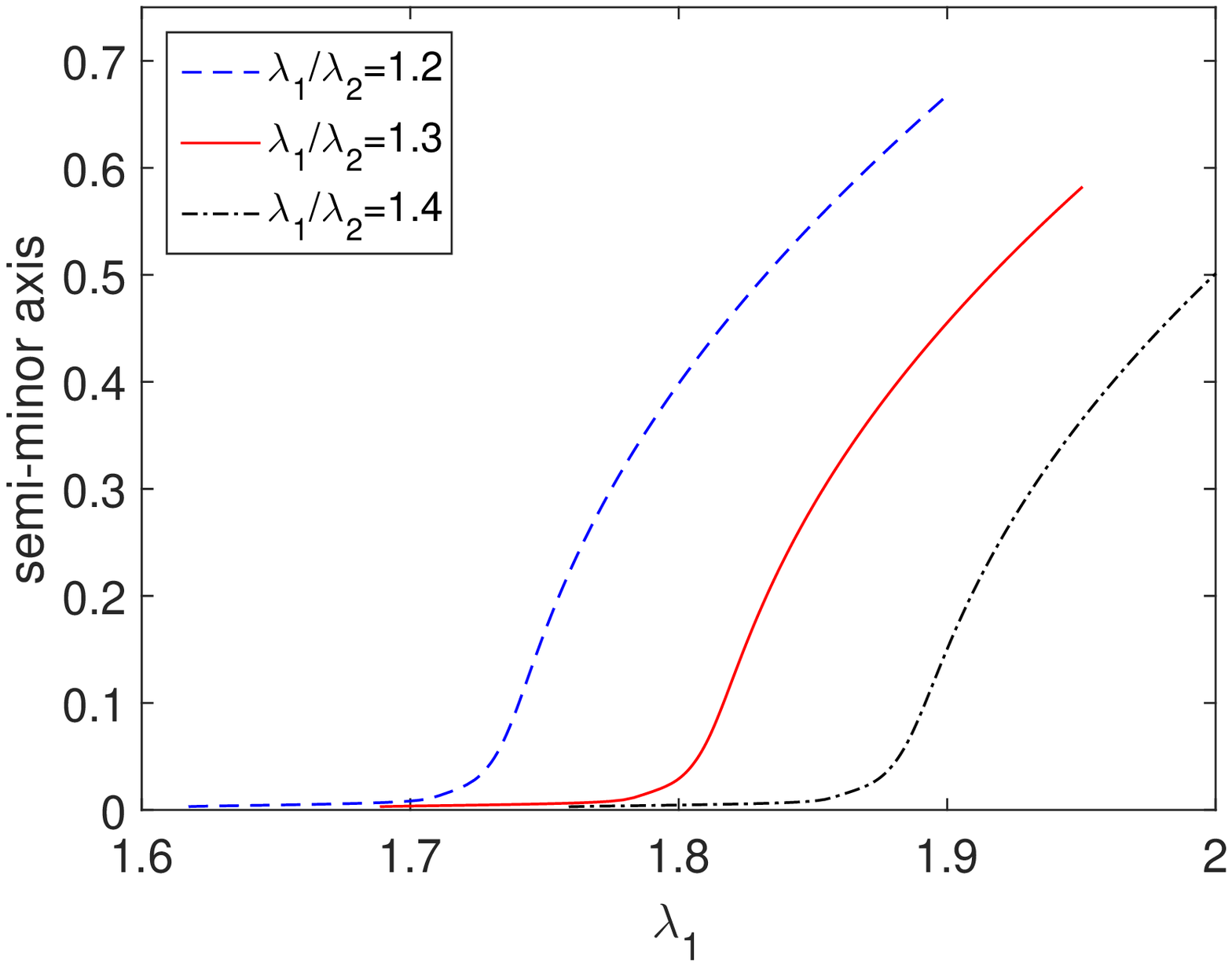}
}
\caption{The semi-major and semi-minor axes of the numerical cavity formed.}
\label{fig:lambda_crit_nonsymmetry}
\end{figure}

Taking $\Omega_{(\varepsilon,\gamma)}=\Omega_{(10^{-3},10^{-1})}$, $\Omega_{(10^{-4},10^{-2})}$,
$\Omega_{(10^{-5},10^{-3})}$ as the reference configuration
respectively, setting the oval boundary data $\lambda_1 \gamma=1.67$ and $\lambda_2 \gamma=1.42$
(see Table~\ref{tab:fit}),
for $N=16$ fixed, the errors of the corresponding non-radially-symmetric numerical solutions
$\mathbf{U}^{NM}$ are shown in with Figure~\ref{fig:subdomain_nonsym}, where the numerical
solution with $M=32$ is taken as the exact solution. It is clearly seen that
high precision numerical results can also be obtained in a neighborhood of a defect
with rather small $M$ in non-radially-symmetric case.

\begin{figure}[H]
\centering
\subfigure[non-radially-symmetric, energy error]{
\includegraphics[width=7.2cm,height=5.4cm]{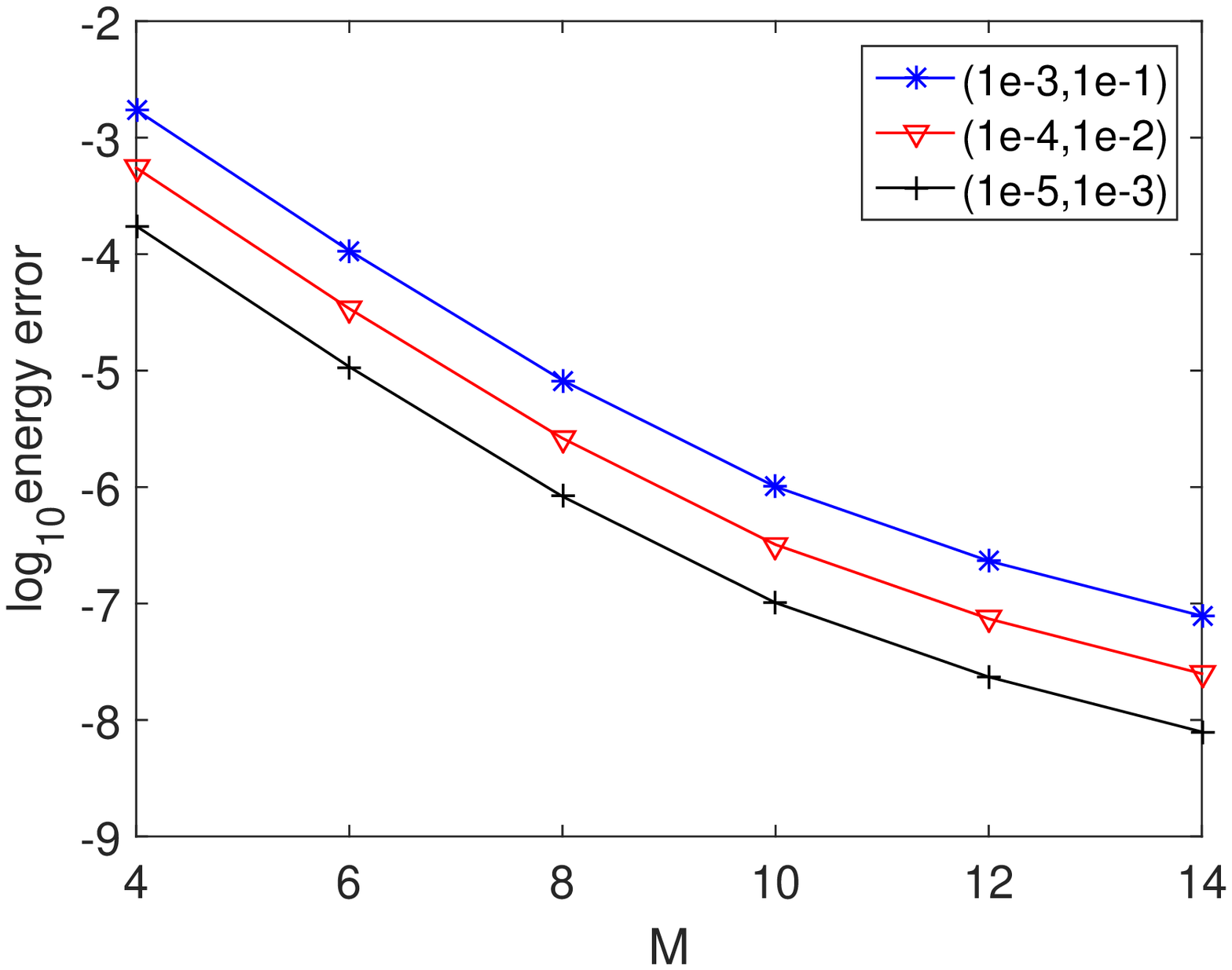}
}
\subfigure[semi-major axis error]{
\includegraphics[width=7.2cm,height=5.4cm]{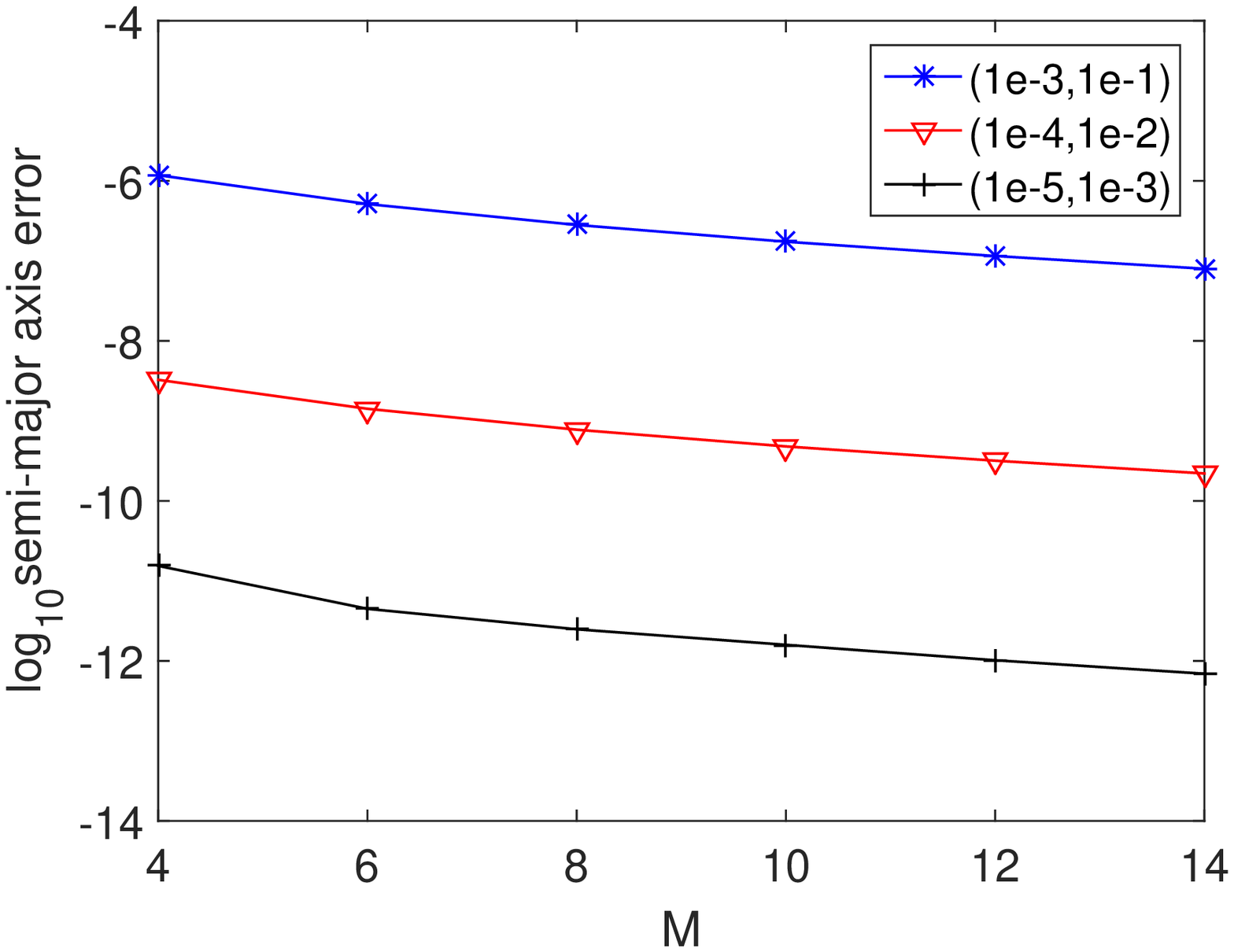}
}
\caption{The convergence behavior on $\Omega_{(\varepsilon,\gamma)}$ with small $\gamma$.}
\label{fig:subdomain_nonsym}
\end{figure}


\bibliographystyle{unsrtnat}

\addcontentsline{toc}{chapter}{\bibname}

\end{document}